\documentclass[10pt,reqno]{amsart}
\usepackage{amsmath}
\usepackage{amssymb}
\usepackage{amsthm}
\usepackage{graphicx}
\usepackage[table]{xcolor}
\usepackage[inline]{enumitem}
\usepackage{hyperref}
\usepackage{cleveref}
\usepackage[normalem]{ulem}

\setlist[enumerate]{label={\upshape(\roman*)}}

\newtheorem{theorem*}{Theorem}[section]
\newtheorem{theorem}{Theorem}[section]
\newtheorem{lemma}[theorem]{Lemma}
\newtheorem{corollary}[theorem]{Corollary}
\newtheorem{proposition}[theorem]{Proposition}

\theoremstyle{remark}
\newtheorem{definition}[theorem]{Definition}
\newtheorem{example}[theorem]{Example}

\newtheorem{remark}[theorem]{Remark}

\DeclareMathOperator{\rank}{rank}
\DeclareMathOperator{\diag}{diag}

\newcommand{\spans}{\operatorname{span}}
\newcommand{\range}{\operatorname{ran}}
\newcommand{\SC}{\mathcal{S}}

\newcommand{\be}{\begin{equation}\label} 
\newcommand{\ee}{\end{equation}}
\newcommand{\bq}{\begin{equation*}}
\newcommand{\eq}{\end{equation*}}
\newcommand{\ba}{\begin{align*}}
\newcommand{\ea}{\end{align*}}
\newcommand{\bp}{\begin{proof}}
\newcommand{\ep}{\end{proof}}
\newcommand{\bL}{\begin{lemma}\label}
\newcommand{\eL}{\end{lemma}}
\newcommand{\bP}{\begin{proposition}\label}
\newcommand{\eP}{\end{proposition}}
\newcommand{\bC}{\begin{corollary}\label}
\newcommand{\eC}{\end{corollary}}
\newcommand{\bT}{\begin{theorem}\label}
\newcommand{\eT}{\end{theorem}}
\newcommand{\bTT}{\begin{theorem*}\label}
\newcommand{\eTT}{\end{theorem*}}
\newcommand{\bR}{\begin{remark}\label}
\newcommand{\eR}{\end{remark}}
\newcommand{\bD}{\begin{definition}\label}
\newcommand{\eD}{\end{definition}}

\newcommand{\bE}{\begin{example}\label}
\newcommand{\eE}{\end{example}}

\newcommand{\Sc}{\mathcal{S}}

\usepackage{amsrefs}
\theoremstyle{remark}

\newtheorem{lem}{Lemma}[section]
	\newtheorem{rem}[lem]{\textit{Remark}}
\title{\textbf{Automatic selfadjoint-ideal semigroups for finite matrices
}}

\begin{document}
\author{Sasmita Patnaik*}
\address{Department of Mathematics and Statistics, Indian Institute of Technology, Kanpur 208016 (INDIA)}
\email{sasmita@iitk.ac.in}
\author{Sanehlata}
\address{Department of Mathematics and Statistics, Indian Institute of Technology, Kanpur 208016 (INDIA)}
\email{snehlata@iitk.ac.in}
\author{Gary Weiss**}
\address{Department of Mathematics, University of Cincinnati, Cincinnati, OH 45221-0025 (USA)}
\email{gary.weiss@uc.edu}

\thanks{*Supported by Science and Engineering Research Board, Core Research Grant 002514. \quad\\
  **Partially supported by Simons Foundation collaboration grants 245014 and 636554.}


\maketitle	
	
\begin{abstract}

The notion of automatic selfadjointness of all ideals in a multiplicative semigroup of 
the bounded linear operators on a separable Hilbert space $B(\mathcal H)$ arose in a 2015 discussion with Heydar Radjavi who pointed out that $B(\mathcal H)$ and the finite rank operators $F(\mathcal H)$ possessed this unitary invariant property which category we named SI semigroups (for automatic selfadjoint ideal semigroups). Equivalent to the SI property is the solvability, for each $A$ in the semigroup, of the \textit{bilinear} operator equation $A^* = XAY$ which we believe is a new connection relating  
the semigroup theory with the theory of operator equations.

We found in our earlier works in the subject that even at the basic level of singly generated semigroups, the investigation of SI semigroups led to interesting algebraic and analytic phenomena when generated by rank one operators, normal operators, partial and power partial isometries, subnormal-hyponormal-essentially normal operators, and weighted shift operators; and generated by commuting families of normal operators.

In this paper, we focus on a separate $M_n(\mathbb C)$ treatment for singly generated SI semigroups 
that requires studying 
the solvability of the bilinear matrix equation $A^* = XAY$ in a multiplicative semigroup of finite matrices.
This separate focus is needed because the techniques employed in our earlier works we could not adapt to finite matrices. 
In this paper we find that for certain classes of generators, being a partial isometry is equivalent to generating an SI semigroup. Such classes are: degree $2$ nilpotent matrices, weighted shifts, and non-normal Jordan matrices. For the key tools used to establish these equivalences, we developed a number of necessary conditions for singly generated semigroups to be SI for the very general classes: nonselfadjoint matrices, nonzero nilpotent matrices, nonselfadjoint invertible matrices, and Jordan blocks. 
We also show, for a nonselfadjoint matrix generator in an SI semigroup, the matrix being a partial isometry is equivalent to having norm one. And as an aside, we also prove necessary generator conditions for the SI property when generated by matrices with nonnegative entries.
\end{abstract} 


\section{introduction}
In operator theory, an important area of study is the solvability of linear and bilinear operator equations in the algebra of operators $B(\mathcal H)$ acting on finite or infinite-dimensional Hilbert spaces; for instance, operator equations of the form $AX = Y$, $AX - XB = Y$, $AX - YB = C$, or $B = XAY$ (see \cite{RR}, \cite{FW}, \cite[Section 8]{SO} and the references therein). 

In particular, the need to study the solvability of the
\textit{bilinear} operator equations $A^* = XAY$ in a \textit{multiplicative semigroup} arose in a 2015 discussion with Heydar Radjavi who observed that this operator equation is solvable in the multiplicative semigroups $B(\mathcal H)$ and $F(\mathcal H)$, respectively, or equivalently and more to his point, every multiplicative ideal inside each of these is automatically selfadjoint (i.e., contains all its adjoints). (Here $F(\mathcal H)$ denotes the algebra of finite rank operators.) And he asked which multiplicative subsemigroups of $B(\mathcal H)$ and $F(\mathcal H)$ share this selfadjointness property for all its multiplicative ideals? Going forward unless we specify otherwise, we view $B(\mathcal H)$ as a multiplicative semigroup and refer to its multiplicative subsemigroups as merely semigroups. 

The first and third author of this paper in \cite[Lemma 1.8]{PW21} noticed that the solvability of $A^* = XAY$ for each $A$ in a semigroup is equivalent to the automatic selfadjointness of all the principal ideals generated by $A$ for each $A$ in the semigroup, and a semigroup with this property (all ideals or equivalently just principal ideals) we called a selfadjoint-ideal semigroup (SI semigroup, for short). See Definitions (\ref{D1}-\ref{D6}) and Terminology below. For instance, from this algebraic point of view, one can say that $B(\mathcal H)$ and $F(\mathcal H)$ are SI semigroups. But on the other hand, the singly generated selfadjoint semigroup $\Sc(T,T^*) $ generated by $T$, with $T$ a nonselfadjoint
operator and $||T|| < 1$, is never an SI semigroup (see \cite[Example 1.23]{PW21}). It was natural then to ask if it is possible to characterize, in $B(\mathcal H)$, all its SI semigroups in some way. We found this SI property interesting because it turned out to be a unitary invariant of semigroups in $B(\mathcal H)$ and hence a useful tool in distinguishing between them up to unitary equivalence, and sometimes in determining their simplicity (i.e., whether or not they have no proper multiplicative ideals).  

In \cite{PW21}, we initiated a systematic study of SI semigroups by first investigating the SI property of the singly generated selfajdoint semigroups generated by a single $T \in B(\mathcal H)$ and denoted by $\mathcal S(T, T^*)$ (all finite products of $T$ and $T^*$) since the study of the general structure of selfadjoint semigroups seemed unexplored. For a selfadjoint operator $T$, $\Sc(T, T^*)$ consists of only powers of $T$ and hence is automatically an SI semigroup because every ideal is automatically selfadjoint. So we restricted our study to singly generated selfadjoint semigroups generated by a nonselfadjoint operator. We focused on the attributes of $T$ that are necessary or necessary and sufficient to guarantee that their singly generated semigroups  
$\Sc(T, T^*)$  possess the SI property. We obtained characterizations 
(necessary and sufficient conditions) for non-simple singly generated SI semigroups and simple singly generated semigroups generated by normals, partial isometries and among non-normals, rank-ones (for a summary of those characterizations see \cite{PW21} succeeding Theorem 3.20). Recently in \cite{PWS} joint with the second author, we expanded our investigation 
to the study of singly generated SI semigroups generated by unilateral weighted shift operators, hyponormal operators, essentially normal operators, and generated by commuting families of normal operators. We studied how the SI property for a singly generated semigroup impacts the generator in terms of its spectral density when it is hyponormal and when it is essentially normal \cite[Section 3]{PWS}. 
As stated in \cite[2nd paragraph preceding Theorem 1.21]{PW21},
we would like to alert the reader that the fact that semigroups may not be closed under scalar multiples, 
 even in the simplist cases of rank-one $T$,
 substantially complicated our proof of the characterization of SI semigroups
$\Sc(T,T^*)$ generated by $T$ \cite[Theorem 3.16]{PW21}. 

Here in this program we focus on higher rank operators in $F(\mathcal H)$ which is equivalent to  
studying 
the SI characterization question in $M_n(\mathbb C)$, viewed as a multiplicative semigroup, by first investigating the SI property of its singly generated selfadjoint subsemigroups with generator $T$  for special classes of generators, since the study of the SI property of singly generated selfadjoint semigroups generated by a general $T$ seemed intractable and remains open to us. Our original aim here was to focus on Jordan matrices (Section \ref{S4}), 
but found this to require a number of preliminary somewhat general results on certain classes of matrices (Sections \ref{S2}-\ref{S4}), which we briefly discuss and summarize below.

The techniques employed in our earlier study of singly generated SI semigroups in $B(\mathcal H)$ (for operators acting on infinite-dimensional Hilbert space) we could not adapt to the finite-dimensional case, hence our  separate focus herein on the SI property of semigroups in $M_n(\mathbb C)$. Our approach to this study is to investigate the special forms associated with a finite matrix motivated by the following reasoning. From linear algebra we recall that a finite matrix is unitarily equivalent to a triangular matrix (and strictly upper or lower triangular if and only if nilpotent); and a finite matrix is similarity equivalent to a matrix which is either in rational cannonical form or in Jordan cannonical form. We know from \cite[Theorem 1.21]{PW21} that if $A$ and $B$ are unitarily equivalent then $\Sc(A,A^*)$ is SI if and only if $\Sc(B, B^*)$ is SI, but \cite[Remark 3.22]{PW21} shows that the SI property of semigroups is not a similarity invariant property.
But when $A$ is a triangular matrix or $A$ is in the rational cannonical form, the computations involved in solving the bilinear matrix equations $A^* = XAY$ in $\Sc(A, A^*)$ seemed intractable. 
So despite the fact that Jordan matrices are not totally general (in the unitarily equivalent sense), they are direct sums of Jordan blocks which are special triangular matrices, and with obtaining some preliminary somewhat general results, we were able to completely characterize those Jordan matrices that singly generate an SI semigroup. 

To summarize, in this paper we investigate singly generated SI semigroups of finite matrices by offering a new perspective on the interplay between the study of a particular matrix equation with the structure of the generating matrix of the SI semigroup. 
And in some cases, this leads to a classification of non-simple SI semigroups and simple semigroups.
The main result of this paper is Theorem \ref{jordan} where we provide a characterization of those singly generated SI semigroups $\Sc(T, T^*)$ generated by a Jordan matrix (see the definition in Section \ref{S4}, third paragraph). This also provides an alternate characterization of power partial isometry. (A power partial isometry is a partial isometry with all its powers also partial isometries.) When $T$ belongs to either the class of nilpotent matrices of degree $2$ or the class of weighted shift matrices, here also we provide an alternate SI characterization of partial isometry: $T$ is a partial isometry if and only if $\Sc(T, T^*)$ is SI (see Corollaries \ref{cor2.1}-\ref{corollary2.6}). 
Moreover, as we know partial isometries have norm one but not all norm one matrices are partial isometries, in the SI environment they are equivalent. That is, we prove that under the SI property of $\Sc(T, T^*)$ generated by a nonselfadjoint matrix $T$, if $||T|| =1$ then $T$ is a partial isometry (Theorem \ref{p}).

We found that the general problem of characterizing SI semigroups $\Sc(T, T^*)$ for a general matrix $T$ is still open but the complications we encountered in dealing with the case of Jordan matrices illustrates the value of sparsifying a matrix via unitary equivalence because of the unitary invariance of SI semigroups and the motivation for focusing first on Jordan matrices.


\subsection{Preliminaries}\label{S1}
\textbf{Definitions (\ref{D1}-\ref{D6}) and Terminology}

We recall below the general $B(\mathcal H)$ definitions and terminology from \cite{PW21}, but instead for finite matrices.

\bD{D1}
A semigroup $\Sc$ in $M_n(\mathbb C)$ is a subset closed under multiplication. 
 A selfadjoint semigroup $\Sc$ is a semigroup also closed under adjoints,  i.e., $\Sc^* := \{T^* \mid T \in \Sc\} \subset \Sc$. 
\eD
 
\bD{D3}
 An ideal $J$ of a semigroup $\mathcal S$ in $M_n(\mathbb C)$ is a subset of $\Sc$ closed under products of operators in $\Sc$ and $J$. That is, $XT, TY \in J$ for $T \in J$ and $X, Y \in \mathcal S$. And so also $XTY \in J$. 
 \eD

The next definition is new to the field of multiplicative $B(\mathcal H)$-semigroups, motivated by Radjavi and first published in \cite{PW21}.

   \bD{D4} A selfadjoint-ideal (SI) semigroup $\Sc$ in $M_n(\mathbb C)$ is a semigroup for which every ideal $J$ of $\Sc$ is closed under adjoints, i.e., $J^{*} := \{T^{*} \mid T \in J\} \subset J.$  
 \eD
Because this selfadjoint ideal property in Definition \ref{D4} concerns selfadjointness of all ideals in a semigroup, we call these semigroups  selfadjoint-ideal semigroups (SI semigroups for short). 

\vspace{.2cm}

\textit{Semigroups generated by $\mathcal A \subset M_n(\mathbb C)$}
\bD{D5}
The semigroup generated by a set $\mathcal A \subset  M_n(\mathbb C)$, denoted by $\mathcal S(\mathcal A)$, is the intersection of all semigroups containing $\mathcal A.$ Also define $\mathcal A^*:= \{A^* | A \in \mathcal A\}$.
\eD
For short we denote by $\Sc(T)$ the semigroup generated by $\{T\}$ (called generated by $T$ for short).
It should be clear for the semigroup $\mathcal S(\mathcal A)$ that Definition \ref{D5} is equivalent to the semigroup consisting of all possible words of the form $A_1A_2\cdots A_k$ where $k \in \mathbb N$ and $A_i \in \mathcal A$ for each $1\leq i \leq k$. 
 
 \bD{D6}
The selfadjoint semigroup generated by a set $\mathcal A \subset M_n(\mathbb C)$ denoted by $\Sc(\mathcal A \cup \mathcal A^{*})$ or $\Sc(\mathcal A,\mathcal A^{*})$, is the intersection of all selfadjoint semigroups containing  $\mathcal A \cup \mathcal A^{*}$. 
Let $\Sc(T, T^*)$ denote for short $\Sc(\{T\}, \{T^*\})$ and call it the singly generated selfadjoint semigroup generated by $T$.
 \eD
 
It is clear that $\Sc(\mathcal A, \mathcal A^{*})$ is a selfadjoint semigroup. Moreover, it is clear that Definition \ref{D6} conforms to the meaning of $\Sc(\mathcal A \cup \mathcal A^*)$ in terms of words discussed above. That is, it consists of all words of the form $A_1A_2\cdots A_k$ where $k \in \mathbb N$ and $A_i \in \mathcal A \cup \mathcal A^{*}$ for each $1\leq i \leq k$. 

The focus of this paper is the investigation of the singly generated SI semigroups $\Sc(T, T^*)$ generated by $T \in M_n(\mathbb C)$. So, we provide a description of the elements of $\Sc(T, T^*)$ here (\cite[Proposition 1.6]{PW21}).

For $T \in B(H)$, the semigroup $S(T, T^*)$ generated by the set $\{T, T^*\}$ is given by 
\begin{equation}\label{singly generated semigroup list}
 S(T, T^*) = \{T^n, {T^*}^n, \Pi_{j=1}^{k}T^{n_j}{T^*}^{m_j},  (\Pi_{j=1}^{k}T^{n_j}{T^*}^{m_j})T^{n_{k+1}}, \Pi_{j=1}^{k}{T^*}^{m_j}T^{n_j}, (\Pi_{j=1}^{k}{T^*}^{m_j}T^{n_j}){T^*}^{m_{k+1}}\}
\end{equation}
where $n \ge 1,\,  k\ge1,\, n_j, m_j \ge 1\, \text{for}\, 1 \le j \le k, ~\text{and}~n_{k+1}, m_{k+1} \geq 1$.
The product $\Pi_{j= 1}^{k}$ in the semigroup list is meant to denote an ordered product. 
Indeed, this follows directly by taking $\mathcal A = \{T\}$.

Alternatively, $\Sc(T, T^*)$ consists of: words only in $T$, words only in $T^*$, words that begin and end in $T$, words that begin with $T$ and end with $T^*$, and words that begin with $T^*$ and end with $T$ and words that begin and end with $T^*$.

\section{$\SC(T, T^*)$ characterizations and necessary conditions on $T$ \\
for the SI property and simplicity}\label{S2}

In \cite[Section 3]{PW21} we obtained a characterization 
(i.e., a set of necessary and sufficient conditions depending on the class in which $T$ resides) 
for the SI property of semigroups $\SC(T, T^*)$ generated by a rank-one operator $T$; and in some cases the SI property characterized the simplicity of $\SC(T, T^*)$. (A summary of this complete classification is provided in \cite[before Remark 3.21]{PW21}.) The various levels of difficulty and limited techniques at our disposal  complicated there our approach to  this case by case characterization for the SI semigroup $\SC(T, T^*)$ in this simplest  case of rank-ones. 
We began our study of general rank-one operators (instead of considering matrix forms) with the hope that we could extend our results to higher ranks, which turned out not to be the case. 
 So for us to make progress in the study of SI semigroups $\SC(T, T^*)$ generated by finite rank operators,  
here we reduce 
the study of SI semigroups $\SC(T, T^*)$ generated by finite ranks to the study of SI semigroups $\SC(T, T^*)$ generated by finite matrices using the following observation. 
And then we develop differing methods for various classes of generating matrices. 

Observe that for $T\in \mathcal{F}(\mathcal{H})$ (the set of finite rank operators on a Hilbert space $\mathcal H$), $T^*$ is also finite rank (via an argument using the polar decomposition), and hence the subspace $\mathcal H_n = T\mathcal H + T^*\mathcal H$ is a finite dimensional reducing subspace for $T$ with $T$ unitarily equivalent to  $T_n\oplus 0$ for some $T_n\in \mathcal{B}(\mathcal{H}_n)$. Furthermore, as every operator on a finite dimensional space is unitarily equivalent to an upper triangular matrix, $T_n$ is unitarily equivalent to an operator in $B(\mathcal H_n)$ whose matrix is upper triangular and hence $T$ has a basis in which its matrix has form $T_n \oplus 0$ and is upper triangular. In other words, $T\in \mathcal{F}(\mathcal{H})$ is unitarily equivalent to $T_n\oplus 0$ where the matrix representation of $T_n$ with respect to some orthonormal basis in $\mathcal H_n$ is upper triangular. 

On unitary invariance of the SI property, recall that if $T$ and $S$ are unitarily equivalent, then $\SC(T, T^*)$ is SI if and only if $\SC(S, S^*)$ is SI (a special case of \cite[Theorem 1.21]{PW21}). In particular, for $T$ a finite rank operator, as discussed in the previous paragraph, $T$ is unitarily equivalent to $T_n\oplus 0$ where the finite matrix representation of $T_n$ with respect to some orthonormal basis in $\mathcal H_n$ is upper triangular. Hence we have that $\SC(T,T^*)$ is SI if and only if $\SC(T_n, T^*_n)$ is SI. Recall also \cite[Lemma 1.9]{PW21} that in general, $\SC(T, T^*)$ having the SI property is equivalent to solving the equation $W^* = XWY$ for every word $W$ in $T$ and $T^*$ for 
$X, Y \in \SC(T, T^*) \cup \{I\}$. 
So herein our study of the SI property of $\SC(T,T^*)$ we reduce to the study of the SI property of $\SC(T_n, T^*_n)$. That is, we consider upper triangular matrix representations of $T_n$ to study the SI property of $\SC(T_n, T^*_n)$ and the solvability of the aforementioned Bilinear equation.

So the study of when $\SC(T,T^*)$ possesses the SI property for $T$ finite rank is reduced to the case when $T \in M_n(\mathbb C)$ and our first result in this direction is a necessary condition for $\SC(T, T^*)$ to be an SI semigroup when $T\in M_n(\mathbb{C})$ is nonselfadjoint. 
(When $T$ is selfadjoint, $\SC(T, T^*) = \{T^n \mid n \geq 1\}$ is clearly automatically SI. See \cite[Remark 1.13(i)-(ii)]{PW21} for a detailed discussion of this case.)

In summary, going forward here and in the last section we continue our investigation of SI semigroups $\SC(T, T^*)$ generated by a finite rank operator $T$ beyond rank-one by focusing on finite matrix cases.
In this section, we first provide a necessary condition for $\SC(T, T^*)$ to be an SI semigroup when $T$ is nonselfadjoint (Theorem \ref{thm 2.1} below). 
Then for the class of nilpotent matrices (equivalently those with unitarily equivalent strictly upper triangular matrix representations), we provide some new connections between the SI property of $\SC(T, T^*)$ and partial isometries. In particular, we give a necessary condition for $\SC(T, T^*)$ to be SI when generated by a nilpotent matrix (Corollary \ref{cnilpotent}); and as a consequence, we provide partial answers in Corollaries \ref{cor2.1}-\ref{corollary2.6} to \cite[Question 2.7]{PW21}: Characterize which partial isometries $T$ have their generated semigroups $\SC(T, T^*)$ possessing the SI property, and among those determine which 
possess the stronger property of simpleness. 
Earlier Popov--Radjavi had proved the alternate (to definition) characterization of power partial isometry: that an operator $T$ is a power partial isometry if and only if $\SC(T, T^*)$ consists of \textit{only} partial isometries \cite[Proposition 2.2]{HeyPop}. By combining this result with \cite[Corollary 1.15]{PW21}, Patnaik--Weiss in \cite[Remark 2.3]{PW21} proved that if $T$ is a power partial isometry, then $\SC(T, T^*)$ is SI. And regarding the converse, in Remark \ref{R2} below, using Corollary \ref{corollary2.6}, we obtain that the converse also holds if $T$ is a unilateral weighted shift matrix.

\begin{remark}\label{Remark1}
We note here that the next Propositions \ref{pnil}-\ref{prop2.3} are proved for $M_n(\mathbb C)$, but a simple argument shows they also hold for finite rank operators in $B(\mathcal H)$ because, as said earlier, every finite rank operator is unitarily equivalent to a finite rank operator whose matrix representation is the direct sum of a finite-dimensional upper triangular matrix and an infinite-dimensional zero matrix. We believe that Propositions \ref{pnil}-\ref{prop2.3} are well known, but we have presented them here for completeness and because they will be employed repeatedly in some later parts of this section. 
\end{remark}

	\begin{proposition}\label{pnil}
		For $T \in M_n(\mathbb{C})$ a nonzero nilpotent matrix, one has
		 $$\range T^*\not \subset \range T \quad \text{and} \quad \range T\not \subset \range T^*.$$
	\end{proposition}

\begin{proof} 
It is well-known that every square matrix $T$ is unitarily equivalent to an upper (or lower) triangular matrix $A$, and nilpotency being a unitary invariant, it is easily verified that the diagonal of $A$ must be zero.
And it is straightforward to verify that the range non-inclusions in the statement of the proposition are also unitarily invariant. Hence it suffices to prove the proposition for $A = [a_{ij}]$ a nilpotent upper triangular  matrix, that is, $a_{ij}=0$ for all $i\geq j$.

To obtain $\range A^* \not\subset \range A$,
let $k$ be the maximum index such that the column $A^{*}e_k \neq 0$ (such a $k$ exists as $A\neq 0$). 
Since $A$ has diagonal $0$,  $A^*e_n = \bar{a}_{nn}e_n = 0$ and hence $k<n$. 
Then also $a_{ij}=0$ for each $i> k$ and for all $j$.  
This implies that the span of the $A$ columns, that is, $\range A  \subset \spans \{e_1, \cdots, e_k\}$. 
Since $A^*$ is strictly lower triangular and $A^{*}e_k \neq 0$, $A^*e_k = \bar{a}_{k,k+1}e_{k+1} + \bar{a}_{k,k+2}e_{k+2}+...+ \bar{a}_{k,n}e_{n}$ where at least one of the coefficients is nonzero. Therefore $\range A^* \not\subset \range A$ or equivalently, $\range T^* \not\subset \range T$.

To obtain $\range T \not\subset \range T^*$, since the adjoint of a nilpotent matrix is nilpotent, apply the previous case to $T^*$.
\end{proof}

\begin{proposition}\label{prop2.2}
	For $A,B\in M_n(\mathbb{C})$,  $\rank(AB)\leq \min\{\rank A, \rank B\}$.
	More generally, for $A_1, \cdots, A_k \in M_n(\mathbb{C})$, $\rank(A_1\cdots A_k) \leq \min\{\rank A_1,  \cdots, \rank A_k\}$.
	\end{proposition}
\begin{proof} Recall that for $T\in M_n(\mathbb{C}), \rank T=\rank T^*$. Since $\range AB \subset \range A$, $\rank(AB)\leq \rank A$. Also, $\range(AB)^*= \range(B^*A^*)\subset \range B^*$, which implies that $\rank(AB) \leq \rank(B)$. Therefore, $ \rank(AB)\leq \min\{\rank A, \rank B\}$. By induction this holds for any finite product of matrices in $M_n(\mathbb{C})$.
	\end{proof}

\begin{proposition}\label{prop2.3}
For $A$ a nonzero nilpotent finite matrix, one has $$\rank A^2<\rank A.$$
\end{proposition}

\begin{proof} 
Since $\range A^2 \subset  \range A$ one has $\rank A^2 \le \rank A$. 
But if $\rank A^2 = \rank A$ then $\range A^2 = \range A$ implying by a simple induction that 
$\range A^n = \range A$ for all $n \ge 1$, contradicting nonzero nilpotency. 
\end{proof}


We are now ready to prove a necessary condition 
involving kernels and partial isometries for $\Sc(T, T^*)$ to be an SI semigroup.

\begin{theorem}\label{thm 2.1}
Let $T\in M_n(\mathbb{C})$ be a nonselfadjoint matrix. If $\Sc(T, T^*)$ is an SI semigroup, then
$$\text{either } \ker T=\ker T^2 \text { or } T \text{ is a partial isometry.}$$	
\end{theorem}
\begin{proof}
	Suppose $\Sc(T, T^*)$ is an SI semigroup. Then $(T)_{\Sc(T, T^*)}$ is selfadjoint. Therefore, $T^*=XTY$ for some $X,Y\in\Sc(T, T^*)\cup\{I\}$ where either $X$ or $Y \neq I$ (since $T$ is nonselfadjoint). Recalling by Definition \ref{D5} that all members of $\Sc(T, T^*)$ are words in $T$ and $T^*$, if $XTY$ contains any powers higher than one of $T$ or $T^*$, then by Proposition~\ref{prop2.2}, 
\ one obtains $\rank(XTY)\leq \rank T^2$ or $\rank{T^*}^2$. So $T^*=XTY$ together with the fact that $\rank T^*=\rank T$ implies that $\rank T= \rank(XTY)\leq \rank(T^2)$. Also, $\rank(T^2)\leq \rank T$ (also Proposition~\ref{prop2.2}). Hence  $\rank(T^2)= \rank T$. By the rank-nullity theorem, $\dim(\ker T)=\dim(\ker T^2)$. Since $\ker T \subset \ker T^2$ and $\dim(\ker T)=\dim(\ker T^2)$, one has $\ker T= \ker T^2$.

If $XTY$ does not contain any higher powers of $T$ or $T^*$ and recalling not both $X,Y$ are the identity operator, then 
by Equation (\ref{singly generated semigroup list}) and avoiding all cases where $T$ or $T^*$ have powers higher than one,
	$$T^* = XTY\in \{T, T^*, (TT^*)^k,(TT^*)^kT,(T^*T)^k,(T^*T)^kT^*\}$$ 
for some $k\geq 1$. 
Note that since $T$ is not selfadjoint and $XTY$ contains no higher powers than one, but must contain a $T$, it follows that $XTY$ cannot be of the first, second, third or the fifth form in the above display. 
So either $T^*=(TT^*)^kT$ or $XTY=(T^*T)^kT^*$ for some $k \geq 1$. 
In the fourth case $XTY=(TT^*)^kT$, one has $T^* = XTY=(TT^*)^kT$ and
 by right multiplying $T$ on both sides, one obtains, $T^*T=(TT^*)^kT^2$ which implies that $\ker T^2 \subset \ker T^*T$. But as is well-known and obvious to prove, $\ker T^*T=\ker T$. 
Therefore $\ker T^2\subset \ker T$ which, along with the obvious reverse inclusion, implies that $\ker T=\ker T^2$, proving the theorem in this case. 
On the other hand in the sixth case, if $XTY=(T^*T)^kT^*$, then $T^*=(T^*T)^kT^*$,  and so by right multiplying $T$ on both sides, one obtains, $T^*T=(T^*T)^{k+1}$. Therefore by the spectral theorem, $T^*T$ is a projection, or equivalently, $T$ is a partial isometry.
\end{proof}

Recalling again that all square matrices are unitarily equivalent to some upper triangular matrix, the next result about nilpotent matrices interests us because these are the ones whose unitarily equivalent upper triangular forms, by a direct computation, are those that are strictly upper triangular.

\begin{corollary}\label{cnilpotent}
	Let $T\in M_n(\mathbb{C})$ be a nonzero nilpotent matrix. If $\Sc(T,T^*)$ is an SI semigroup, then $T$ is a partial isometry.
\end{corollary}
\begin{proof}
	Since $T$ is a nonzero nilpotent matrix, $T$ is not selfadjoint. From Proposition~\ref{prop2.3}, $\rank T^2< \rank T$. Therefore, by the rank-nullity theorem, $\dim(\ker T)< \dim(\ker T^2)$ and so $\ker T\neq \ker T^2$.
	Since $\Sc(T,T^*)$ is an SI semigroup, by Theorem \ref{thm 2.1}, $T$ is a partial isometry.
\end{proof}

The converse of Corollary \ref{cnilpotent} does not hold in general, see Example \ref{E1} below where we provide a class of such nilpotent matrices for which the converse does not hold. But if the nilpotency degree is $2$, then Corollary~\ref{cor2.1} below proves the converse and so yields a characterization of those semigroups $\Sc(T,T^*)$ generated by a nilpotent matrix of degree $2$ that are SI.

\begin{example}\label{E1}
	Let $T\in M_n(\mathbb{C})$ be a nilpotent partial isometry with nilpotency degree equal to $3$  (hence nonselfadjoint).  If $||T^2||<1$ 
(for instance 
$\begin{pmatrix}
0 & 1/\sqrt 2 & 0 & 1/\sqrt 2\\
0 & 0 & 1 & 0 \\
0&0&0&0\\
0&0&0&0
\end{pmatrix}$
is such a nilpotent of degree 3 partial isometry),
then $\Sc(T,T^*)$ is not SI. Indeed, without loss of generality we may assume that $T$ is a strictly upper triangular matrix 
	and it suffices to show that $(T^2)_{\Sc(T,T^*)}$ is not selfadjoint. Suppose otherwise that $(T^2)_{\Sc(T,T^*)}$ is selfadjoint. Then ${T^*}^2=XT^2Y$ for some $X,Y\in \Sc(T,T^*)\cup \{I\}$ but not both can be the identity (as ${T^*}^2\neq T^2$ because $T^2$ is also strictly upper triangular). Moreover, Proposition~\ref{pnil} applied to ${T^*}^2=XT^2Y$ and its adjoint equation $T^2 = Y^*{T^*}^2X^*$ implies via its contrapositive that 
$X\neq I$ and  $Y\neq I$ so both $X,Y \in \Sc(T,T^*)$. And since $T^3 = 0$, so $X,Y \ne T$. Also note that neither $X$ nor $Y$ has any higher than one power of $T$ or $T^*$, since otherwise from the assumptions $||T||=1$ (as $T$ is a partial isometry), and $||T^2||<1$ and the assumption ${T^*}^2=XT^2Y$, one obtains either $||X||<1$ or $||Y||<1$.
	This then implies that $||T^2||=||{T^*}^2|| = ||XT^2Y||<||T^2||$, a contradiction. 

Then with the additional hypothesis that $T$ is a partial isometry 
(with equivalences: $T = TT^*T$, $T^* = T^*TT^*$, $T^*T$ is a projection, or $TT^*$ is a projection  \cite[Corollary 3 of Problem 127]{Hal82}), so $\Sc(T,T^*)$ consisting of all words in $T$ and $T^*$, those words without higher powers than one fall into the four categories of alternating  $T$ and $T^*$: starting and ending with each of $T$ or $T^*$.
Those starting and ending with $T$, clearly via the identity $T = TT^*T$, reduce to $T$, except the case $X$ or $Y = T$ which was ruled out above;
those starting and ending with $T^*$, clearly via the identity $T^* = T^*TT^*$, reduce to $T^*$; 
those starting with $T$ and ending with $T^*$, clearly via  $TT^*$ being a projection, reduce to $TT^*$;
and those starting with $T^*$ and ending with $T$, clearly via  $T^*T$ being a projection, reduce to $T^*T$. That is, 
$$X, Y \in \{T^*, TT^*, T^*T\}.$$ 
And moreover since $T^3 = 0$, it becomes clear via the contrapositive, because $T^2, T^{*2} \ne 0$ with ${T^*}^2=XT^2Y$, that 
 $$X\in \{T^*, TT^*\} \text { and } Y\in \{ T^*, T^*T\}.$$ 

By considering these four cases 
$(X,Y) = (T^* \text{ or } TT^*, T^* \text{ or } T^*T)$ 
we can now prove the unsolvability of ${T^*}^2=XT^2Y$. 	 
Indeed, when $X=T^*$, one has ${T^*}^2=XT^2Y=T^*T^2Y$, which by left multiplying $T^*$ one obtains $0={T^*}^3={T^*}^2T^2Y$. So in the case $X=T^*$, $Y = T^*$, one has $0 = {T^*}^2T^2T^*$. 
Then right multiplying by $T$ and substituting $T = TT^*T$ (an equivalent characterization of partial isometry) we obtain $0=({T^*}^2T^2T^*)T= ({T^*}^2T)(TT^*T) = {T^*}^2T^2 = ||T^2||^2 = 0$ implying 
$T^2=0$, contradicting the $T$ nilpotency of degree $3$.  
On the other hand in the case $Y=T^*T$, substituting $T = TT^*T$ one has 
${T^*}^2=XT^2Y=T^*T^2T^*T=T^*T^2$ and taking adjoints yields $T^2 ={T^*}^2T$ implying that 
$\range T^2\subset \range {T^*}^2$, a contradiction to Proposition~\ref{pnil} applied to the nonzero nilpotent operator $T^2$. 
Therefore, when $X = T^*$ and $Y \in \{ T^*, T^*T\}$, then ${T^*}^2 \neq XT^2Y$.  

For the cases when $X = TT^*$ and $Y \in \{ T^*, T^*T\}$, then substituting again $T = TT^*T$, 
either ${T^*}^2 = XT^2Y =  TT^*T^2T^* = T^2T^*$ or ${T^*}^2 = (TT^*T)(TT^*T) = T^2$. 
The former equation ${T^*}^2  = T^2T^*$, after right multiplying by $T^*$, as before implies that $T^2 = 0$, against $T$ nilpotent degree $3$ (or one can apply Proposition~\ref{pnil} applied to the nonzero nilpotent operator $T^2$ for a contradiction); 
and the latter equation implies that $T^2$ is selfadjoint, against $T^2$ a nonzero nilpotent matrix. 
Therefore, neither of the required equations hold.
This completes the proof of the nonsolvability of the equation ${T^*}^2=XT^2Y$ 
in $\Sc(T, T^*) \cup \{I\}$, thereby proving that $\Sc(T, T^*)$ is not SI.
	 \end{example}
	
\begin{remark}\label{R1}
The condition $||T^2||<1$ in the above example cannot be dropped. For instance, consider $T:=\begin{bmatrix} 
	0 & 1 & 0\\
	0 & 0 & 1\\
	0 & 0 & 0
	\end{bmatrix}$. Then $\Sc(T,T^*)$ is SI because $T$ is a power partial isometry \cite[Corollary1.15]{PW21} but its square has norm one.
\end{remark}

\begin{remark}\label{R2}
In Example \ref{E1}, it took a $4\times4$ matrix example to produce a degree 3  nilpotent partial isometry whose square has norm strictly less than one.
In fact, after reducing the problem to a strictly upper triangular matrix, with some work it can be shown that there is no such $3\times 3$ example. 
\end{remark}

For the nilpotent degree 2 case, we have a partial isometry characterization for  $\Sc(T,T^*)$ being SI.

	\begin{corollary} \label{cor2.1}
		Let $T\in M_n(\mathbb C)$ be a nilpotent matrix of degree $2$. 
Then $\Sc(T,T^*)$ is SI if and only if $T$ is a partial isometry.
	\end{corollary}
\begin{proof}
	$\Rightarrow$: This is Corollary \ref{cnilpotent}.

	$\Leftarrow$:
	Suppose $T$ is a partial isometry. Since $T^2=0$, $T$ is a power partial isometry. Therefore $\Sc(T,T^*)$ is SI by \cite[Corollary 1.15]{PW21}.
\end{proof}

We next turn to weighted shifts. That is, $T\in M_n(\mathbb{C})$ for which $Te_j = \alpha_je_{j+1}$ for $1 \leq j \leq n-1$ and $Te_n = 0$.

\begin{corollary}\label{corollary2.6}
	Let $T\in M_n(\mathbb{C})$ be a weighted shift matrix with weights $\{\alpha_j\}^{n-1}_{j=1}$. 
Then $\Sc(T,T^*)$ is SI if and only if $T$ is a partial isometry.
\end{corollary}
\begin{proof}
	Clearly SI and partial isometry are unitary invariant properties. Since $T$ is unitarily equivalent to the matrix with weights $\{|\alpha_j|\}^{n-1}_{j=1}$ (\cite[Problem 89]{Hal82}), without loss of generality we may assume $\alpha_j \geq 0$.
	Suppose $\Sc(T,T^*)$ is SI. Then by Corollary \ref{cnilpotent}, it follows that $T$ is a partial isometry. Conversely, if $T$ is a partial isometry, this equivalent to $T^*T = \diag(\alpha^2_1, \cdots, \alpha^2_{n-1}, 0)$ being a projection. This further implies that the nonzero $\alpha_{j}$'s must be equal to $1$. 
Then by a straightforward computation one sees that ${T^*}^k$ and $T^k$ are respectively upper and lower $k$ diagonals with weights products of the $\alpha_j$'s and that the diagonal matrix ${T^*}^kT^k$ is also a projection for each $k\geq2$.
Therefore, $T$ is a power partial isometry and hence by \cite[Corollary 1.15]{PW21}, $\Sc(T,T^*)$ is SI.
\end{proof}

\begin{theorem}\label{SI PI PPI}
For weighted shifts, the following are equivalent.
\begin{enumerate}[label=(\roman*)]
\item $\Sc(T,T^*)$ is SI.
\item $T$ is a partial isometry.
\item $T$ is a power partial isometry.
\end{enumerate}
\end{theorem}

We conclude this section with a theorem on the simplicity of semigroups when generated by an invertible matrix whose application is seen in Corollary \ref{nonsimple} and Theorem \ref{theorem1}. In general, a group is always simple (since having inverses, every nonzero multiplicative ideal contains the identity, hence is the whole group), but a semigroup need not be simple  (\cite{PW21} abounds with examples, and as well herein any of the non SI semigroups cannot be simple). In Theorem \ref{inverse}, we prove that under the SI property of $\Sc(T, T^*)$, the invertibility of a nonselfadjoint $T$ implies that $\Sc(T, T^*)$ is a group and hence simple. That is, for the class of semigroups $\Sc(T, T^*)$ with $T$ nonselfadjoint and invertible, SI and simplicity are equivalent.

\begin{theorem}\label{inverse}
		Let $ T\in M_n(\mathbb C)$ be a nonselfadjoint invertible matrix. Then $ \Sc(T,T^*)$ is an SI semigroup if and only if $\Sc(T,T^*)$ is simple.
	\end{theorem}
\begin{proof}
	If $\Sc(T,T^*)$ is simple, then it has no ideals and so is vacuously SI.
 
But to prove that SI semigroups $\Sc(T, T^*)$ are simple requires work. First, we recall the semigroup list Equation (\ref{singly generated semigroup list}):
	
\noindent	$\Sc(T,T^*) =  \{T^n, {T^*}^n, \Pi_{j=1}^{k}T^{n_j}{T^*}^{m_j},  (\Pi_{j=1}^{k}T^{n_j}{T^*}^{m_j})T^{n_{k+1}}, \Pi_{j=1}^{k}{T^*}^{m_j}T^{n_j}, (\Pi_{j=1}^{k}{T^*}^{m_j}T^{n_j}){T^*}^{m_{k+1}}\},$ 

\noindent where $n \ge 1,\,  k\ge1,\, n_j, m_j \ge 1\, \text{for}\, 1 \le j \le k~\text{and}~m_{k+1}\geq 1, n_{k+1}\geq 1$.
	
Now suppose $\Sc(T,T^*)$ is SI. Then $(T)_{\Sc(T,T^*)}$ is a selfadjoint-ideal, in particular. So, $T^*=XTY$ for some $X,Y\in\Sc(T,T^*)\cup\{I\}$, where both $X,Y$ are not identity becasue $T$ is nonselfadjoint. 
Observe that the only word in the above semigroup list without a $T$ in it is the second term ${T^*}^n$; 
and because $T$ is nonselfadjoint, $XTY$ being $T^*$ cannot take the form $T^n$ for $n = 1$. 
Therefore $XTY$ may only take at least one of the remaining forms, all of which by observation take one of the following four possible forms. 
	$$\text{(i)}\, XTY = T^*WT, \quad 
		\text{(ii)}\, XTY = TWT^*, \quad
		\text{(iii)}\, XTY = T^*WT^*,\quad \text{ and}\quad
		\text{(iv)}\, XTY = TWT,$$		
where $W\in \Sc(T,T^*)\cup \{I\}$.

Case (i). If $XTY = T^*WT$, then $T^* = T^*WT$. Since $T$ is not selfadjoint, $W \neq I$.  Left multiplying equation $T^* = T^*WT$ by the inverse ${T^*}^{-1}$, one obtains 
$I=WT \in \Sc(T,T^*)$ and so $T^{-1} = W\in \Sc(T,T^*)$. 
Hence in this case, since the inverse of the generator is also in the semigroup, every word in $T$ and $T^*$ has an inverse in the semigroup, that is, every member of $\Sc(T,T^*)$ has its inverse in the semigroup and so the semigroup contains the identity $I$. And so the semigroup is a group, hence it is simple. 
Case (ii) can be handled similarly as in Case (i). 

Case (iii). If $XTY=T^*WT^*$, then $T^*=T^*WT^*$. Here also, $W \neq I$. Otherwise, $T^*={T^*}^2$ and the invertibility of $T^*$ implies that $T^*=I$ which is selfadjoint contradicting the nonselfadjointness of $T$. 
Then since $W \neq I$, again left multiplying by the inverse ${T^*}^{-1}$ in the equation $T^*=T^*WT^*$, one obtains, $I=WT^*$. So, ${T^*}^{-1} = W \in \Sc(T,T^*)$. Therefore, $\Sc(T,T^*)$ forms a group and hence is simple.

Case (iv). If $XTY=TWT$, then $T^*=TWT$. We first claim that if $T^*=TWT$, then $I \in \Sc(T,T^*)$.
 Since  $T^*=TWT$, one has $T^*T^{-1}=WT$ and  $T^{-1}T^* = TW$.  Therefore, $T^*T^{-1}$ and $T^{-1}T^*\in \Sc(T,T^*)$. 
Then $T{T^*}^{-1}=(T^{-1}T^*)^* \in \Sc(T,T^*)$, since $\Sc(T,T^*)$ is a selfadjoint semigroup. Hence, $I=(T^*T^{-1})(T{T^*}^{-1})\in \Sc(T,T^*)$.
And since $T^*=TWT$, so $T^{-1}T^*T^{-1} = W \in \Sc(T,T^*)$. 
But then also one has  $T=T^*W^*T^*$ implying $T{T^*}^{-1}=T^*W^* \in \Sc(T,T^*)$, 
where the latter inclusion holds because $W \in \Sc(T,T^*) \cup \{I\}$ so $W^* \in \Sc(T,T^*) \cup \{I\}$ again by the selfadjointness of $\Sc(T,T^*)$, hence $T{T^*}^{-1}=T^*W^* \in \Sc(T,T^*)$. 
Therefore 
 $(T^*T^{-1})^{-1} = TT^{*-1}\in \Sc(T,T^*)$. 
And finally because $T^{-1}T^*T^{-1} \in \Sc(T,T^*)$ one has 
 $T^{-1} = (T^{-1}T^*T^{-1})((T^*T^{-1})^{-1}) \in \Sc(T,T^*)$ so all words in $T$ and $T^*$ have inverses in $\Sc(T,T^*)$ and $I \in \Sc(T,T^*)$. 
So again $\Sc(T,T^*)$ is a group and hence is simple.
\end{proof}

An immediate application of this Theorem \ref{inverse} combined with our first paper on the subject \cite[Summary preceding Remark 3.21]{PW21}
classifies all SI semigroups  $\SC(T, T^*)$ generated by $T \in M_2(\mathbb C)$. 
We caution that this result may not hold for rank two operators in $F(\mathcal H)$.

\begin{corollary} Every matrix in $M_2(\mathbb C)$ has rank one or two.  For a complete classification of when $\SC(T, T^*)$ is SI:

In the rank one case, see \cite[Summary preceding Remark 3.21]{PW21} for a complete classification.

In the rank two case, the nonselfadjoint matrix is invertible and hence its generated semigroup is SI if and only if it is simple.
\end{corollary}

\section{Characterization of SI semigroups $\SC(T, T^*)$ generated by a 
Jordan matrix $T$}\label{S4}

In our earlier study of the SI property of semigroups, we recall from \cite[Theorem 1.21]{PW21} that the SI property of a semigroup is unitary invariant, but not similarity invariant. In particular, if $A$ and $B$ are similar matrices (but not unitarily similar) and $\SC(A,A^*)$ is an SI semigroup, then $\SC(B,B^*)$ may not be an SI semigroup (see \cite[Remark 3.22]{PW21}). 
From linear algebra, we know that every finite matrix $A$ is unitarily equivalent to an upper triangular matrix $B$. But when $A$ and $B$ are unitarily equivalent, $\SC(A, A^*)$ is an SI semigroup if and only if $\SC(B, B^*)$ is an SI semigroup, so one would naturally be inclined instead to investigate the SI property of $\SC(B, B^*)$. 

 For invertible matrices $B$, $\SC(B, B^*)$ is SI if and only if it is simple (Theorem \ref{inverse})). But for a non-invertible $B$, the study of the SI property of $\SC(B, B^*)$ where $B$ is an upper triangular matrix, even in the case of $3 \times 3$ matrix, seemed to us computationally complicated, and in higher dimensions, even more intractable. 

Nevertheless, we know from linear algebra that every matrix is similar (via an invertible, not necessarily unitary) to a matrix in Jordan form and rational cannonical form (i.e., direct sum of Jordan and rational cannonical blocks).  
These similarity equivalent matrices are central to the study of linear algebra in part because some of the crucial properties of matrices are similarity invariant, as for instance, rank, nullity, determinant, and trace. 
But because rational cannonical matrices are much more complicated than Jordan matrices to compute with in solving the necessary bilinear equations, in this section we investigate the SI nature of Jordan matrices.
We found Jordan matrices much more approachable albeit nontrivial.

In this section, we focus on a characterization of SI semigroups $\SC(T,T^*)$ generated by a Jordan matrix $T$.
By a Jordan matrix $T$ (i.e., Jordan form) we will always mean a block diagonal matrix 
$T=\oplus_{i=1}^k J_{m_i}(\lambda_i)$, where $k\geq 1$ and 
$J_{m_i}(\lambda_i)$ is Jordan block ($m_i \times m_i$ block corresponding to eigenvalue $\lambda_i$ which has all the diagonal entries equal to $\lambda_i$ and all supdiagonal entries equal to one (see below Equation (\ref{eq1'})).
 Note that a Jordan matrix is a nonnormal matrix whenever at least one of its blocks is larger than $1 \times 1$. 
And a Jordan block is invertible if and only if its $\lambda \ne 0$.

Our focus in this section is to obtain a necessary and sufficient condition on the Jordan matrix $A$ for $\Sc(A, A^*)$ to be SI. Theorem \ref{thm 2.1} provides in the nonselfadjoint case an either/or necessary condition 
($\ker A = \ker A^2$ or $A$ is a partial isometry) for 
$\Sc(A, A^*)$ to be SI. Also, since $\Sc(A, A^*)$ being SI implies each $\Sc(J_{m_i}(\lambda_i),J^*_{m_i}(\lambda_i))$ is SI for $1\leq i\leq k$ (Proposition \ref{prop2}), then using Proposition \ref{prop3}, 
a necessary condition for each nonselfadjoint invertible Jordan block to generate an SI semigroup is that $|\lambda_i| = 1$ for each nonzero $\lambda_i$. So in order to obtain a sufficient condition on $A$ for $\Sc(A, A^*)$ to be SI, we must restrict to the unit circle the nonzero $\lambda$'s that appear in the Jordan blocks of the Jordan matrix $A$. We first investigate the selfadjoint semigroup generated by each Jordan block and establish several results that lead to a sufficient condition for $\Sc(A, A^*)$ to be an SI semigroup. We then prove  in Theorem \ref{jordan} the necessary and sufficient condition that Jordan matrix $A$ be a power partial isometry in order for 
$\Sc(A, A^*)$ to be an SI semigroup.

For the study of the SI property of $\Sc(T,T^*)$ generated by a Jordan matrix $T$, 
motivated by Theorem \ref{thm 2.1}, we split the class of Jordan matrices into two cases: (i) $\ker T\neq \ker T^2$ and (ii) $\ker T= \ker T^2$. Since a Jordan matrix is the direct sum of Jordan blocks and its SI status is determined by the SI status of its Jordan blocks (Proposition \ref{prop2}), we will first consider the special case when $T$ is a single nonselfadjoint Jordan block, i.e., $T=J_m(\lambda)$ with 
$m\geq 2,~\text{and in the special case when}~ |\lambda|=1$ 
 we will prove that for such a matrix $T$, $\Sc(T,T^*)$ is not an SI semigroup. Later on, we will use this special case to obtain a characterization of SI semigroups $\Sc(T,T^*)$ generated by a Jordan matrix $T$. We may occasionally write $J_m$ for $J_m(\lambda)$ and $\mathcal{W}$  for 
a word $\mathcal{W}(T,T^*)$ in $T,T^*$ whenever $\lambda$ and the matrix $T$ is clear from the context, respectively. First we make a few observations about $J_N(\lambda)$. 
Firstly, $1 \times 1$ Jordan blocks are simply those that have entry $\lambda$, the only normal Jordan blocks. For higher dimensions nonnormality is easily verified.

\vspace{.2cm}

\noindent\textbf{Observations:} For  any $N\geq 2$ the Jordan block matrix $J_{N}(\lambda)$ is given by: \begin{equation}\label{eq1'} 
J_{N}(\lambda)= \begin{bmatrix}
\lambda & 1 & 0&\cdots & 0\\
0 & \lambda & 1 & \cdots & 0\\
\vdots & \vdots & \vdots & \ddots& 1\\ 
0 & 0 & 0& \cdots  &\lambda\\
\end{bmatrix}_{N\times N}
\end{equation}

Let $W$ be the constant weighted shift matrix for $\lambda \ne 0$ defined as:

\begin{equation}\label{eq2}
We_i:=\begin{cases}
(1/\bar{\lambda})\, e_{i+1} \hspace{.2cm} \text{if}~ 1\leq i \leq N-1\\
0 \hspace{.9cm} \text{if}~i=N
\end{cases} 
\end{equation}
Then, one can rewrite $J_{N}(\lambda)$ as 
\begin{equation}\label{eq1}
J_{N}(\lambda)= \begin{bmatrix}
\lambda & 1 & 0&\cdots & 0\\
0 & \lambda & 1 & \cdots & 0\\
\vdots & \vdots & \vdots & \ddots& 1\\ 
0 & 0 & 0& \cdots  &\lambda\\
\end{bmatrix}_{N\times N}=\lambda(I+W^*) 
\end{equation}

Note that $W^N={W^*}^N=0$ and $W^{N-1}={W^*}^{N-1}\ne 0$. Let $T:= I+W^*$, then $J_N(\lambda)=\lambda T$ from Equation (\ref{eq1}). We rewrite $T$ this way and work with $W$ 
instead because it is computationally more convenient and useful for characterizing the SI semigroup $\Sc(T, T^*)$.

We next show via the binomial expansion that an arbitrary word in $T$ and $T^*$ has a certain form in terms of a certain kind of polynomial in $W$ and $W^*$.

\begin{equation}\label{eq3}
T^n=(I+W^*)^n=I+\tbinom{n}{1}W^*+\tbinom{n}{2}{W^*}^2+\cdots + \tbinom{n}{k_n}{W^*}^{k_n}\hspace{.4cm}\text{for} ~n\geq 1
\end{equation} 
where  $k_n$ is a positive integer depending on $n$ such that $k_n=n$ for 
$n<N-1$ and $k_n=N-1$ for $n\geq N-1$ (because ${W^*}^N=0$). 
So also, \begin{equation}\label{eq4}
{T^*}^m=(I+W)^m=I+\tbinom{m}{1}W+\tbinom{m}{2}{W}^2+\cdots + \tbinom{m}{k_m}{W}^{k_m}\hspace{.4cm}\text{for} ~m\geq 1\end{equation}
where $k_m=m$ for $m<N-1$ and $k_m=N-1$ for $m\geq N-1$ (as $W^N=0$). Note that all the coefficients of the powers of $W^*$ and $W$ in Equations (\ref{eq3})-(\ref{eq4}) are positive integers, respectively. 
From these Equations (\ref{eq3})-(\ref{eq4}) 
it clearly follows that 
any word $\mathcal{W}(T,T^*)$ in $T$ and $T^*$ is given by:
\begin{equation}\label{eq5}
\mathcal{W}(T,T^*)=I+\mathcal{P}(W,W^*)
\end{equation}
where $I$ is the $N\times N$ identity matrix and 
$\mathcal{P}(W,W^*)$ is a polynomial in $W,W^*$ with no constant term and the coefficient of each  monomial that appears in the polynomial is a positive integer.

For the constant $\lambda$ weighted shift $W$ defined in Equation (\ref{eq2}), the case 
$|\lambda| = 1$ needed for this paper, next we prove the following proposition about words in $W,W^*$:
\begin{proposition}\label{prop1}
Let $W$ be a constant $\lambda$ weighted shift.
For $\mathcal{W}$ a word in $W,W^*$ with nonzero diagonal, 
its nonzero diagonal entries are all equal to ${|\lambda|}^{-2p}$ for some $p \ge 1$, 
and 1 in the case $|\lambda| = 1$. .
\end{proposition}
\begin{proof}
	It follows from the definition given in Equation (\ref{eq2}) that for any powers $n,m \geq 1$ one has: 
		\begin{equation}\label{eq6}
	W^ne_i=(1/\bar{\lambda}^n)e_{i+n}\quad\text{if}\quad W^ne_i \neq0 
\quad \text{(i.e., $1\leq n \leq N-1$ and $1 \le i \le N-n$).}
	 \end{equation}
Also for any power $m\geq 1$ one has:
	 \begin{equation}\label{eq7}
	{W^*}^me_i=(1/\lambda^m)e_{i-m}\quad \text{if}\quad {W^*}^me_i \neq 0 
\quad \text{(i.e., $1\le m \le N-1$ and $m + 1 \le i \le N$).}
	 \end{equation}
Also any word in $W,W^*$ has at least one of the following possible forms by Equation (\ref{singly generated semigroup list}):
	 \begin{equation}\label{eq8}
	 \{W^n, {W^*}^n,  \Pi_{j=1}^{k}{W^*}^{m_j}W^{n_j},
	 (\Pi_{j=1}^{k}{W^*}^{m_j}W^{n_j}){W^*}^{m_{k+1}},\Pi_{j=1}^{k}W^{n_j}{W^*}^{m_j},(\Pi_{j=1}^{k}W^{n_j}{W^*}^{m_j})W^{n_{k+1}}\}
	 \end{equation} 
		 where $n \ge 1,\,  k\ge1,\, n_j, m_j \ge 1\, \text{ for }\, 1 \le j \le k \text{ and } n_{k+1},m_{k+1}\geq 1$. It is clear as weighted shifts or their adjoints, or from Equations (\ref{eq6})-(\ref{eq7}), that for $n\geq 1,~~W^n~~\text{and}~~{W^*}^n$ have their main diagonal entries all equal to zero. By hypothesis, since $\mathcal W$ has its main diagonal entries not all zero, so $\mathcal{W}$ must have one of the last four forms in above Equation (\ref{eq8}).
	  Suppose $\mathcal{W}$ has the the third form, i.e., $\mathcal{W}=\Pi_{j=1}^{k}{W^*}^{m_j}W^{n_j}$ with at least some $1 \le i \le N$ with  $\left<\mathcal{W}e_i, e_i\right> \neq 0$. Then
	 \begin{equation}\label{eq9}
	 0\neq \mathcal{W}e_i=\Pi_{j=1}^{k}{W^*}^{m_j}W^{n_j}e_i
	 \end{equation}
	 Since $\mathcal{W}e_i\neq 0,~\text{one has}~{W^*}^{m_k}W^{n_k}e_i\neq 0$, therefore $W^{n_k}e_i \neq 0$ and ${W^*}^{m_k}e_{i+n_k} \neq 0$, so by Equations (\ref{eq6})-(\ref{eq7}) one obtains,
	 \begin{equation*}
	 0\neq{W^*}^{m_k}W^{n_k}e_i=({1/\bar{\lambda}}^{n_k}){W^*}^{m_k}e_{i+n_k}
=({1/\bar{\lambda}}^{n_k}\lambda^{m_k})e_{i+(n_k-m_k)} 
	 \end{equation*}
with the nonzeroness dictating $1\le i+(n_k-m_k) \le N$ and the other three constraints in Equations (\ref{eq6})-(\ref{eq7}). 

And notice the coefficient remains fixed at $({1/\bar{\lambda}}^{n_k}\lambda^{m_k})$ for all such $i$. 
	 Then proceeding to peel off terms,  Equation (\ref{eq9}) becomes, for the first equality when $k \ge 2$ and for the second equality when $k \ge 3$, 
	 \begin{align*}
 	 0\neq \mathcal{W}e_i &= 
({1/\bar{\lambda}}^{n_k}\lambda^{m_k})\Pi_{j=1}^{k-1}{W^*}^{m_j}W^{n_j}e_{i+(n_k-m_k)} \\
&=({1/\bar{\lambda}}^{n_k}\lambda^{m_k})\Pi_{j=1}^{k-2}{W^*}^{m_j}W^{n_j}({W^*}^{m_{k-1}}W^{n_{k-1}}e_{i+(n_k-m_k)}).
	 \end{align*}
That is,
\begin{equation}\label{eq11}
0\neq \mathcal{W}e_i =({1/\bar{\lambda}}^{n_k}\lambda^{m_k})\Pi_{j=1}^{k-2}{W^*}^{m_j}W^{n_j}({W^*}^{m_{k-1}}W^{n_{k-1}}e_{i+(n_k-m_k)})
\end{equation}

and again ${W^*}^{m_{k-1}}W^{n_{k-1}}e_{i+(n_k-m_k)}\neq 0$, so using again Equations (\ref{eq6})-(\ref{eq7}) one obtains,
	$${W^*}^{m_{k-1}}W^{n_{k-1}}e_{i+(n_k-m_k)}
={(1/\bar{\lambda}}^{n_{k-1}}\lambda^{m_{k-1}})e_{i+(n_k-m_k)+(n_{k-1}-m_{k-1})}$$
and as before these subscripts from the nonzeroness are dictated by the contraints in  Equations (\ref{eq6})-(\ref{eq7}). Then	
substituting this value in Equation (\ref{eq11}) one has:
	$$0 \ne \mathcal{W}e_i={(1/\bar{\lambda}}^{(n_k+n_{k-1})}\lambda^{(m_k+m_{k-1})})\Pi_{j=1}^{k-2}{W^*}^{m_j}W^{n_j}e_{i+(n_k-m_k)+(n_{k-1}-m_{k-1})}.$$
	Continuing this successively, after $k$-steps, one finally obtains,
	\begin{equation}\label{eq12}
	\mathcal{W}e_i={(1/\bar{\lambda}}^{p}\lambda^{q})e_{i+p-q}
	\end{equation}
where $p:=\sum_{j=1}^{k}n_j$	and $q:=\sum_{j=1}^{k}m_j$ are the sum of powers of $W$ and the sum of powers of $W^*$ that appear in $\mathcal{W}$, respectively. Since $\left<\mathcal{W}e_i,e_i\right> \neq 0$, it follows that
$$0\neq \left<\mathcal{W}e_i,e_i\right>={(1/\bar{\lambda}}^{p}\lambda^{q})\left<e_{i+p-q},e_i\right>,$$
hence $\left<e_{i+p-q},e_i\right> \neq 0$. Therefore, $i+p-q=i$, or equivalently, $p=q$ and hence for all such $i$,
$$\left<\mathcal{W}e_i,e_i\right>={|\lambda|}^{-2p}$$ This completes the proof for the case when $\mathcal{W}$ has third form.

 For the other cases when $\mathcal{W}$ is in the other forms in the list given by Equation (\ref{eq8}), one can similarly compute the expressions for $\mathcal{W}e_i$ using Equations (\ref{eq6})-(\ref{eq7}) to obtain the same conclusion given in Equation (\ref{eq12}). Therefore, all the nonzero diagonal entries of the main diagonal of $\mathcal W$ are equal to ${|\lambda|}^{-2p}$.
	\end{proof}
\begin{remark}\label{rem}
	Let $\mathcal{P}(W,W^*)$ be a polynomial in $W,W^*$ with positive integer coefficients and with no constant term. 
Then when $\lambda \in \mathbb{S}^1$, that  
the diagonal entries of the main diagonal of $\mathcal{P}(W,W^*)$ are nonnegative integers follows from
 Proposition~\ref{prop1}. More generally, again from Proposition~\ref{prop1}, for $\lambda \neq 0$, the diagonal entries of the main diagonal of $\mathcal{P}(W,W^*)$ are nonnegative numbers.
  \end{remark}
Recall that our goal is to prove that for $N\geq 2~\text{and}~\lambda\in \mathbb{S}^1$, the semigroup $\Sc(T,T^*)$ generated by $J_N(\lambda)$ is not an SI semigroup (see Theorem \ref{theorem1} below). Towards proving this, we first show that for $T:=I+W^*$ where $W$ is defined in Equation (\ref{eq2})), the $(1,1)$-diagonal entry of each member in $(T^n{T^*}^m)_{\Sc(T,T^*)}$, the principal ideal  in $\Sc(T,T^*)$ generated by $T^n{T^*}^m$, for $n, m \geq 1$, is a number strictly greater than one. As a consequence, in Corollary \ref{nonsimple}, we prove that the semigroup $\Sc(T, T^*)$ is not SI when $T= I + W^*$.

\begin{lemma}\label{lemma1}
	For $T=I+W^*$ where $W$ is defined in Equation (\ref{eq2}) with $\lambda \ne 0$ and $n,m \geq 1$, 
each member of $(T^n{T^*}^m)_{\Sc(T,T^*)}$ has its $(1,1)$-diagonal entry 
 is greater than or equal to $1 + nm|\lambda|^{-2}$. 
	\end{lemma}
\begin{proof}
	Using Equations (\ref{eq3})-(\ref{eq4}), we express $T^n{T^*}^m$ in terms of $W,W^*$ as follows:
	\begin{align*}
	T^n{T^*}^m=&\{I+\tbinom{n}{1}W^*+\tbinom{n}{2}{W^*}^2+\cdots\tbinom{n}{k_n}{W^*}^{k_n}\}\{I+\tbinom{m}{1}W+\tbinom{m}{2}{W}^2+\cdots + \tbinom{m}{k_m}{W}^{k_m}\}\\
	&=\{I+n_1W^*+n_2{W^*}^2+\cdots n_{k_n}{W^*}^{k_n}\}\{I+m_1W+m_2{W}^2+\cdots + m_{k_m}{W}^{k_m}\}\\
	&(n_i,m_j~\text{positive integers depending on $n,m$ with}~~ 1\leq i\leq k_n,1\leq j\leq k_m ~\text{and}~~n_1,m_1=n,m)\\
	&=I+nmW^*W+\mathcal{Q}(W,W^*)
	\end{align*}
	where $\mathcal{Q}(W,W^*)$ is some polynomial in $W,W^*$ with no constant term and the coefficients of all nonzero terms in $\mathcal{Q}(W,W^*)$ are positive integers. 
Observe that since $W^*W= |\lambda|^{-2}I_{N-1}\oplus0$, one has for $I+nmW^*W$, its $(1,1)$-diagonal entry equal to $1+nm|\lambda|^{-2}$.
	
	For every $B \in (T^n{T^*}^m)_{\Sc(T,T^*)}$ one has $B=X(T^n{T^*}^m)Y$ for some $X,Y\in\Sc(T,T^*)\cup\{I\}$ \cite[Lemma 1.7]{PW21}. If $X~(\text{or}~ Y)\in \Sc(T,T^*)$, then $X~(\text{or}~ Y)$ is a word in $T,T^*$ and so from Equation (\ref{eq5}), $X~(\text{or}~ Y)$ is of the form $I+\mathcal{P}(W,W^*)$, where $\mathcal{P}(W,W^*)$ is a polynomial in $W,W^*$ with no constant term and with positive integer coefficients. For the cases $X~(\text{or}~Y)=I$, we take that $\mathcal{P}(W,W^*)$ is the zero polynomial. With this convention, one can rewrite $X(T^n{T^*}^m)Y$ as:
	\begin{equation}\label{eq13}
	X(T^n{T^*}^m)Y=(I+\mathcal{P}_1(W,W^*))(I+nmW^*W+\mathcal{Q}(W,W^*))(I+\mathcal{P}_2(W,W^*))
	\end{equation}
	where $\mathcal{P}_i(W,W^*)$ is either the zero polynomial or a polynomial in $W,W^*$ with no constant term and with positive integer coefficients for $i=1,2$, and $\mathcal{Q}(W,W^*)$ is a polynomial in $W,W^*$ with no constant term and which also has positive integer coefficients. Therefore, one can further simplify the expression in Equation (\ref{eq13}) and rewrite it as
	$$	X(T^n{T^*}^m)Y=I+nmW^*W+\mathcal{P}(W,W^*),$$
	where $\mathcal{P}(W,W^*)$ is a ploynomial in $W,W^*$ with no constant term and with positive integer coefficients. From Remark~\ref{rem}, $\mathcal{P}(W,W^*)$ has all the diagonal entries of its main diagonal as nonnegative numbers. Also, we discussed earlier that $I+nmW^*W$ has its $(1,1)$-diagonal entry 
$1 + nm|\lambda|^{-2}$, and hence $I+nmW^*W+\mathcal{P}(W,W^*)$ has its $(1,1)$-diagonal entry greater or equal to $1 + nm|\lambda|^{-2}$. 
Therefore $X(T^n{T^*}^m)Y$ has its $(1,1)$-diagonal entry 
 greater or equal to $1 + nm|\lambda|^{-2}$. Since $B$ is arbitrary,  
this completes the proof. 
	\end{proof}
	
\begin{corollary}\label{nonsimple}
For $T=I+W^*$ where $W$ is defined in Equation (\ref{eq2}) 
 with $\lambda \ne 0$, the semigroup $\Sc(T, T^*)$ is not an SI semigroup.
\end{corollary}	
\begin{proof}
Suppose $\Sc(T, T^*)$ is an SI semigroup. Since $T$ is a nonselfadjoint invertible matrix, it follows from the SI equivalence Theorem \ref{inverse} that $\Sc(T, T^*)$ is simple. This implies that $T$ is contained in every nonzero principal ideal, in particular, in the principal ideal $(TT^*)_{\Sc(T, T^*)}$. But by Lemma \ref{lemma1}, every element of this principal ideal ($XTT^*Y$ with $X,Y \in \Sc(T, T^*) \cup \{I\}$) has its $(1,1)$-diagonal entry bigger than one and the $(1, 1)$-diagonal entry of $T$ 
 is precisely one. So, $T$ is not in this principal ideal, contradicting the simplicity. Therefore, $\Sc(T, T^*)$ is not an SI semigroup.
\end{proof}
	
 We now prove the nonsimplicity of the semigroup $\Sc(A,A^*)$ generated by $A = J_N(\lambda)$, where $N\geq 2$ and $\lambda\in \mathbb{S}^1$.
 \begin{theorem}\label{theorem1}
For any $N\geq 2, |\lambda| \ge 1$ and $A:=J_N(\lambda)$, the semigroup $\Sc(A,A^*)$ is not an SI semigroup.
 \end{theorem}
\begin{proof}
As mentioned in Equation (\ref{eq1}), $A=\lambda(I+W^*)=\lambda T$, where $T=I+W^*$. 
By Theorem \ref{inverse}, since $A$ is invertible, $\Sc(A, A^*)$ possesses the SI property if and only if it is simple. So it suffices to prove nonsimplicity.
	For proving the non-simplicity of $\Sc(A,A^*)$, it suffices to show that $A\notin (AA^*)_{\Sc(A,A^*)}$. Suppose $A\in (AA^*)_{\Sc(A,A^*)}$. Then 
	$$A=X(AA^*)Y$$
	for some $X,Y\in\Sc(A,A^*)\cup\{I\}$ at least one of which is not the identity since $A$ is nonselfadjoint. Replacing $A$ by $\lambda T$ in the above display one obtains:
	\begin{equation}\label{eq14}
	\lambda T={\lambda}^r{\overline{\lambda}}^k(X'TT^*Y')
	\end{equation}
	where $r,k\geq 1, r+k \ge 3 \text{ and}~X',Y'\in\Sc(T,T^*)\cup\{I\}$. Note that $X'TT^*Y'\in (TT^*)_{\Sc(T,T^*)}$ and hence by Lemma~\ref{lemma1} (the case $n,m = 1$), the $(1,1)$-diagonal entry of $X'TT^*Y'$ is 
 greater than or equal to $1 + |\lambda|^{-2}$. 
Hence, it follows that the matrix ${\lambda}^r{\overline{\lambda}}^k(X'TT^*Y')$ in Equation (\ref{eq14}) has its $(1,1)$-diagonal entry with modulus 
greater than or equal to $|\lambda|^p(1+|\lambda|^{-2})$ for some $p \ge 3$, which itself is strictly greater than $|\lambda|$ as $|\lambda|\ge 1$. 
So we arrive at a contradiction as $\lambda T$ has all its diagonal entries equal to $\lambda$. 
Hence, $A \notin (AA^*)_{\Sc(A,A^*)}$ and so $\Sc(A,A^*)$ is not simple.
\end{proof}
We are now ready to investigate the general case of nonselfadjoint Jordan matrices and obtain a power partial isometry characterization of all SI  semigroups $\Sc(A,A^*)$ generated by nonselfadjoint Jordan matrices $A$ (Theorem \ref{jordan}). 
As all but $1 \times 1$ Jordan blocks are obviously nonselfadjoint (even nonnormal as a direct computaton shows), motivated by Theorem \ref{thm 2.1} we divide our approach into studying three cases: $\ker A\neq \ker A^2$, $\ker A = \ker A^2$ and the partial isometry case.  
Towards this, we first characterize partially isometric Jordan matrices in the following proposition. (A partially isometric Jordan matrix is a Jordan matrix which is also a partial isometry.)
Following that we  
investigate the SI property of $\Sc(A, A^*)$ by considering the two cases separately for the Jordan matrix $A$: $\ker A\neq \ker A^2$ and $\ker A = \ker A^2$.
 
 \begin{proposition}\label{prop}
	A partially isometric Jordan matrix $A$ is unitarily equivalent to $U\oplus B$, (any one summand maybe absent) where $U$ is a diagonal unitary matrix and $B$ is a direct sum of shifts.
\end{proposition}
\begin{proof}
Any direct sum is partially isometric if and only if each of its blocks is partially isometric. This is clear using the facts that partial isometries are those operators whose absolute values are projections, and projections are selfadjoint idempotents. Therefore Jordan matrix $A$ is a partial isometry if and only if each of its Jordan blocks is. It is clear (from these characterizations for instance) that a $1 \times 1$ Jordan block is a partial isometry if and only if $|\lambda| = 0~\text{or}~1$, and for $N \ge 2$, the larger shifts are partial isometries (the cases $\lambda = 0$), but for $\lambda \ne 0$,
those Jordan blocks have norms bounded below by some column norms $\sqrt {1 + |\lambda|^2}$ exceeding the norm of nonzero partial isometries which is $1$. In short, reordering the basis, the direct sum of all nonzero $1 \times 1$ Jordan matrices, if any, forms $U$, and the direct sum of the rest (the $1 \times 1$ zero blocks and the larger shifts), if any, forms $B$. 
\end{proof}

The following corollary may be known, but we will need it going forward.

\begin{corollary}\label{cor}
	A partially isometric Jordan matrix $A$ is always a power partial isometry. 
\end{corollary}
\begin{proof}
By Proposition~\ref{prop}, it suffices to show all powers of $U$ and $B$ are partial isometries.
Clearly all powers of unitary operators are unitary and hence partial isometries, as well as powers of zero matrices. And a straightforward computation shows powers of shifts are partial isometries.
And then powers of direct sums of power partial isometries are partial isometries as mentioned in the previous proof.
	\end{proof}
\begin{corollary}\label{cor2}
	For a Jordan matrix $A$ with $\ker A\neq \ker A^2$, 
	$\Sc(A,A^*)$  is an SI semigroup if and only if A  is a partial isometry.
\end{corollary}
\begin{proof}
	Suppose $\Sc(A,A^*)$  is an SI semigroup. Since $\ker A\neq \ker A^2$, by Theorem \ref{thm 2.1}, $A$ is a partial isometry.
	Conversely, if $A$ is a partial isometry, then by Corollary~\ref{cor}, $A$ is a power partial isometry. Therefore, as proved in \cite[Remark 2.4]{PW21}, $\Sc(A,A^*)$  is an SI semigroup.
\end{proof}

We next consider the Jordan matrix case when $\ker A= \ker A^2$. In this case, $A$ may or may not be invertible. Suppose $A$ is not invertible and $\ker A= \ker A^2$. Noninvertibility implies at least one Jordan block has zero eigenvalue. Then the size of a Jordan block corresponding to the zero eigenvalue of $A$ must be one because the presence of a Jordan block $J_m(0)$ (where $m\geq 2$) violates the condition $\ker A= \ker A^2$. Indeed, any Jordan matrix $A$ with the Jordan block $J_m(0)$ (where $m\geq 2$) must have a column $Ae_i$ such that $Ae_i=e_{i-1}$ and $Ae_{i-1}=0$ for some $i\geq 2$. Therefore, $e_i\in \ker A^2$ but $e_i \notin \ker A$. Therefore, by rearranging the Jordan blocks corresponding to zero and nonzero eigenvalues together, $A$ must be unitarily equivalent to $A_1\oplus 0$, where $A_1$ is an invertible Jordan matrix and $0$ is a matrix of size at least one. Then by a straightforward argument one can prove that $\Sc(A,A^*)$  is an SI semigroup if and only if $\Sc(A_1,A_1^*)$ is an SI semigroup. Based on these observations, our SI investigation for the case when the Jordan matrix $A$ is not invertible reduces to the SI investigation for the invertible corner of the Jordan matrix $A$. 
Therefore, when $\ker A= \ker A^2$, we will consider only the invertible Jordan matrices to study the SI property of $\Sc(A,A^*)$, starting with Proposition \ref{prop3}. But first some preliminaries.

\begin{proposition}\label{prop2}
	Let $A=A_1\oplus A_2$ be a block diagonal matrix. If $\Sc(A,A^*)$  is an SI semigroup, then $\Sc(A_i,A_i^*)$ is an SI semigroup for each $i=1,2$. 
\end{proposition}
\begin{proof}
	Suppose $\Sc(A,A^*)$  is an SI semigroup.
Multiplication in the semigroup $\Sc(A, A^*)$ is defined in \cite[Section 4, first paragraph]{PW21} as componentwise products.
For proving that $\Sc(A_1,A_1^*)$ forms an SI semigroup, it suffices to show that $(T_1)_{\Sc(A_1,A_1^*)}$ is a selfadjoint ideal for each $T_1 \in \Sc(A_1,A_1^*)$ 
\cite[Lemma 1.9(i)$\Leftrightarrow$(ii)]{PW21}.  For $T_1 \in \Sc(A_1,A_1^*)$ 
(i.e., any word in $A_1, A^*_1$), choose $T_2 \in \Sc(A_2,A_2^*)$ in that same form, so that $T:=T_1\oplus T_2\in \Sc(A,A^*)$ (for instance, if $T_1=A_1^2A_1^*A_1^3$ then take $T_2=A_2^2A_2^*A_2^3$). Since $\Sc(A,A^*)$ is an SI semigroup, one has $T^*=XTY$ for some $X,Y\in\Sc(A,A^*)\cup\{I\}$.  Using the explicit forms for $X$ and $Y$, i.e., $X = X_1 \oplus X_2$ and $Y = Y_1 \oplus Y_2$ for some $X_1, Y_1 \in \Sc(A_1, A_1^*) \cup \{I_1\}$ and for some $X_2, Y_2 \in \Sc(A_2, A^*_2) \cup \{I_2\}$, where $I = I_1 \oplus I_2$, we rewrite the matrix equation $T^*=XTY$ in block diagonal form as:
	$$T_1^* \oplus T^*_2 = (T_1\oplus T_2)^*=(X_1T_1Y_1)\oplus(X_2T_2Y_2).$$
	Therefore, from the above display one obtains:
	$$T^*_1=X_1T_1Y_1\quad\text{for some}~X_1,Y_1\in \Sc(A_1,A_1^*)\cup\{I_1\}.$$
	This proves the selfadjointness of the ideal $(T_1)_{\Sc(A_1,A_1^*)}$. Since $T_1$ was chosen arbitrarily, $\Sc(A_1,A_1^*)$ is an SI semigroup. Likewise $\Sc(A_2,A_2^*)$ is also an SI semigroup.
\end{proof}

Interestingly the converse of Proposition \ref{prop2} can fail, i.e., if both $\Sc(A_1,A_1^*)$, $\Sc(A_2,A_2^*)$ are SI semigroups, then $\Sc((A_1\oplus A_2),(A_1\oplus A_2)^*)$ may not be an SI semigroup. For instance, 
\begin{example}
	Consider $A_1=\begin{bmatrix}
	0& 1 & 0\\
	0 & 0 & 1\\
	0 & 0 & 0
	\end{bmatrix}$ and $A_2=\begin{bmatrix}
	2 & 0\\
	0 & 2
	\end{bmatrix}$. Since $A_1$ is a power partial isometry, $\Sc(A_1,A_1^*)$ is an SI semigroup  \cite[Corollary 1.15]{PW21} and $A_2$ is a selfadjoint matrix, therefore $\Sc(A_2,A_2^*)$ is automatically  an SI semigroup \cite[Remark 1.13]{PW21}. But for $A := A_1\oplus A_2$, $\Sc(A,A^*)$ is not an SI semigroup. Indeed, suppose otherwise that $\Sc(A,A^*)$ forms an SI semigroup. Then $(A)_{\Sc(A,A^*)}$ is a selfadjoint ideal, in particular. Therefore, $A^*=XAY$ for some $X,Y\in\Sc(A,A^*)\cup\{I\}$ with not both $X$ and $Y$ equal to the identity matrix $I = I_1 \oplus I_2$ as $A$ is not selfadjoint. Suppose $X \neq I$. Since $X$ is a word in $A, A^*$ so $X $ is a direct sum of that same word in $A_1, A_1^*$ and $A_2, A_2^*$. Let $X = X_1 \oplus X_2$ and $Y = Y_1 \oplus Y_2$; and rewriting the matrix equation $A^*=XAY$ in block diagonal form one obtains:
	$$A^*_1\oplus A^*_2=X_1A_1Y_1\oplus X_2A_2Y_2$$
	for some $X_i,Y_i\in \Sc(A_i,A_i^*)\cup \{I_i\}$.  Since $X\neq I$, then $X$ is a word in $\Sc(A, A^*)$ which is a direct sum of that same word in $A_i, A^*_i$ for $i = 1,2$, and  $X_2 \neq I_2$ as $X_2$ is that same word in $A_2, A_2^*$. So from the above display one further obtains:
	$$A_2^*=X_2A_2Y_2$$
	where $X_2,Y_2\in \Sc(A_2,A_2^*)\cup \{I_2\}$ and not both $X_2, Y_2$ equal to the identity $I_2$. But also $\Sc(A_2,A^*_2)=\{A^k_2: k\geq 1\}$, since $A_2$ is selfadjoint \cite[Remark 1.13]{PW21}. Hence, $X_2A_2Y_2=A^k_2$ for some $k\geq 2$ and then from above display one obtains $A_2 = A_2^*=A^k_2$ where $k\geq 2$, which is not possible for our choice of $A_2$. Hence $\Sc(A,A^*)$ is not an SI semigroup.
\end{example}

Continuing our strategy discussed in the paragraph preceding Proposition \ref{prop2}, the invertible case, 
we have
\begin{proposition}\label{prop3}
	For $A\in M_n(\mathbb{C})$ a nonselfadjoint invertible matrix, if $\Sc(A,A^*)$  is an SI semigroup 
then $|\det A|=1$. 
\end{proposition}
\begin{proof}
	Suppose $\Sc(A,A^*)$  is an SI semigroup. Then $(A)_{\Sc(A,A^*)}$ is a selfadjoint ideal. Therefore, $A^*=XAY$ for some $X,Y\in \Sc(A,A^*)\cup \{I\}$, where $X, Y$ cannot both be the identity because $A$ is nonselfadjoint. Applying the determinant one obtains 
$$0 \ne \overline{\det A}=(\overline{\det A})^m (\det A)^n$$ for some $m\geq 0, n\geq 1$. For $m=0$ (this happens when the words $X$ and $Y$ have no $A^*$ term) one must have $n\geq 2$ otherwise $A^* = A$ contradicting nonselfadjointness of $A$. Then one obtains 
	$$|\det A|=|\det A|^{m+n}$$ 
	where $m+n\geq 2$. And since $\det A\neq 0$, one obtains $|\det A|=1$.  
\end{proof}

Hence we have

\begin{corollary}\label{cor3}
	For $A=\oplus_{i=1}^k J_{m_i}(\lambda_i)$ a nonselfadjoint invertible Jordan matrix,
	\begin{center}
	$\Sc(A,A^*)$ is an SI semigroup if and only if $A$ is a unitary matrix 
 (equivalently, all $m_i =1$ with $|\lambda_i| = 1$).
	\end{center}
\end{corollary}
\begin{proof}
If $A=\oplus_{i=1}^k J_{m_i}(\lambda_i)$ is a unitary matrix, then $\Sc(A, A^*)$ is a group 
hence simple and so trivially SI (see paragraph preceding Theorem \ref{inverse} for why group implies simple).
Conversely, suppose $\Sc(A,A^*)$ is an SI semigroup. We will first show that $m_i = 1$ for all $1 \leq i \leq k$. Indeed, since $A=\oplus_{i=1}^k J_{m_i}(\lambda_i)$ and $\Sc(A,A^*)$ is an SI semigroup, it follows from Proposition \ref{prop2} that $\Sc(J_{m_i}(\lambda_i),J_{m_i}^*(\lambda_i))$ is an SI semigroup for each $m_i$.
Suppose there exists $i \geq 1$ for which $m_i \geq 2$. Then $J_{m_i}(\lambda_i)$ is a nonselfadjoint invertible matrix. 
Nonselfadjointness is clear. Invertibility holds because $\det J_{m_i}(\lambda_i) \ne 0$ which follows from the invertibility of $A$ via $0 \ne \det A = \Pi_{j=1}^k \det J_{m_j}(\lambda_j)$.
Since $\Sc(J_{m_i}(\lambda_i),J_{m_i}^*(\lambda_i))$ is an SI semigroup, by Proposition \ref{prop3},  $|\det(J_{m_i}(\lambda_i))|=1$ and so $|\lambda_i|=1$.
But by Theorem \ref{theorem1}, $\Sc(J_{m_i}(\lambda_i),J_{m_i}^*(\lambda_i))$ is not an SI semigroup whenever $|\lambda_i| = 1$, contradicting $\Sc(J_{m_i}(\lambda_i),J_{m_i}^*(\lambda_i))$ possessing the SI property. Therefore, for each $i \geq 1$, $m_i = 1$. This implies that $A$ is a diagonal matrix with diagonal Jordan blocks of size one each equal to its eigenvalue $\lambda_i$, and hence $A$ is a nonselfadjoint normal matrix which is invertible. Additionally, as $\Sc(A, A^*)$ is also SI, so it follows from \cite[Theorem 2.1]{PW21} that $A$ is unitary.  	
	\end{proof}

The conclusion in Corollary \ref{cor3} may not hold if we drop the hypothesis that $A$ is a Jordan matrix as seen from the following example of a nonselfadjoint invertible \textit{nonunitary} matrix with a selfadjoint generated SI semigroup.
\begin{example}
	For $A=\begin{bmatrix}
	0 & 1/2\\
	2& 0
	\end{bmatrix}$ a nonselfadjoint invertible matrix (which is not a Jordan matrix), $\Sc(A,A^*)$ is an SI semigroup, but  $A$ is not a unitary matrix. 
Indeed since $(A^*A)(AA^*) = I = (AA^*)(A^*A)$, 
one has that $A^*$ and $A$ have their inverses in $\Sc(A,A^*)$ and hence all its elements (all words in $A,A^*$) have their inverses in $\Sc(A,A^*)$ which makes it a group, hence simple (see first line of previous proof), and hence SI.
\end{example}

We can now characterize SI semigroups $\Sc(A, A^*)$ generated by a nonselfadjoint Jordan matrix $A$.
This can be viewed as:  among nonselfadjoint Jordan matrices, an alternate SI characterization of power partial isometries, 
or equivalently by Corollary \ref{cor}, of partial isometries (because for Jordan matrices, they are the same class).

In \cite[Remark 1.13(i)-(ii) and Theorem 2.1]{PW21} we characterized SI semigroups $\Sc(A, A^*)$ for $A$ normal. Here we do so for non-normal Jordan matrices $A$. 

\begin{theorem}\label{jordan}(A characterization of SI semigroups generated by Jordan matrices.)\\
For non-normal Jordan matrices $A$,
\begin{center}
$\Sc(A, A^*)$ is an SI semigroup if and only if $A$ is a partial isometry.
\end{center}
\end{theorem}
\begin{proof}
Suppose $\Sc(A, A^*)$ is an SI semigroup. Since $A$ is a nonselfadjoint matrix and $\Sc(A, A^*)$ is SI, by Theorem \ref{thm 2.1}, one has either $\ker A = \ker A^2$ or $A$ is a partial isometry. So if $\ker A \neq \ker A^2$, then $A$ is a partial isometry and then by Corollary \ref{cor}, $A$ is a power partial isometry. 

Next we show that under the SI assumption of $\Sc(A, A^*)$ for a non-normal Jordan matrix $A$, $\ker A \neq \ker A^2$. Indeed, suppose $\ker A = \ker A^2$. We consider the invertible and noninvertible cases separately. If $A$ is not invertible, then the discussion in the paragraph after Corollary \ref{cor2} proves that $A$ is unitarily equivalent to $A_1\oplus 0$, where $A_1$ is an invertible Jordan matrix and $0$ is a matrix of size at least one. Since $\Sc(A,A^*)$  is an SI semigroup so $\Sc(A_1,A_1^*)$ is an SI semigroup (as mentioned in that paragraph, follows by a straightforward argument). Since $A_1$ is a  nonselfadjoint invertible Jordan matrix and the SI property of $\Sc(A_1,A_1^*)$ implies that $A_1$ is a unitary matrix by Corollary \ref{cor3}. Therefore $A$ is unitarily equivalent to $U \oplus 0$ which is a contradiction to the non-normality of $A$. In the case of invertible $A$, it follows again from Corollary \ref{cor3} that $A = U$ where $U$ is a unitary matrix, hence normal which again is a contradiction. Therefore in both the invertible and noninvertible case, the SI property of $\Sc(A, A^*)$ implies that $\ker A \neq \ker A^2$. So, $A$ must be a partial isometry.

Conversely, if $A$ is a partial isometry, then by Corollary \ref{cor}, $A$ is a power partial isometry and then $\Sc(A, A^*)$ is an SI semigroup by \cite[Corollary 1.15]{PW21}. 
\end{proof}

We have now established all the results that are required to obtain a  characterization for the simplicity of semigroups $\Sc(A,A^*)$ generated by $A$ from the class of nonselfadjoint Jordan matrices: 	 
\begin{corollary}\label{simplicity cor}
			For $T \in M_n(\mathbb{C})$ a nonselfadjoint Jordan matrix, one has $$\Sc(A,A^*)~ \text{is simple if and only if $A$ is unitarily equivalent to}~ U\oplus 0,$$
			where $U$ is unitary matrix and $0$ is a zero matrix (the second summand maybe absent).
		\end{corollary}
	\begin{proof}
		Suppose $\Sc(A,A^*)$ is simple. Then $\ker A=\ker A^2$. Indeed, if $\ker A\neq\ker A^2$, then $\ker A \subsetneq \ker A^2$. This implies that $\dim \ker A < \dim \ker A^2$, so by the rank nullity theorem, $\rank A > \rank A^2$. We next show that $\rank A > \rank A^2$ implies that $A\notin (A^2)_{\Sc(A,A^*)}$. 
		Suppose $A\in(A^2)_{\Sc(A,A^*)}$. Then $A=XA^2Y$ for some $X,Y\in\Sc(A,A^*)\cup\{I\}$. 
		Using Proposition \ref{prop2.2} and the fact that  $\rank A > \rank A^2$, we obtain
		$$\rank A=\rank(XA^2Y)\leq \rank A^2 <\rank A,$$
		which is absurd.
	This implies nonsimplicity of $\Sc(A,A^*)$, contradicting the simplicity assumption on $\Sc(A, A^*)$. So $\ker A=\ker A^2$.		
We next consider the invertible and the noninvertible cases separately. If $A$ is invertible, then it follows from Corollary \ref{cor3} that $A$ is a unitary matrix (simple semigroups are automatically SI semigroups). If $A$ is not invertible, then the discussion in the paragraph after Corollary \ref{cor2} proves that $A$ is unitarily equivalent to $A_1\oplus 0$, where $A_1$ is an invertible Jordan matrix and $0$ is a matrix of size at least one. And a straightforward argument proves that $\Sc(A,A^*)$ is simple if and only if $\Sc(A_1,A_1^*)$ is simple. Therefore, when $A$ is not invertible, the simplicity of $\Sc(A,A^*)$ reduces to the simplicity of $\Sc(A_1,A^*_1)$ where $A_1$ is the invertible corner of the Jordan matrix $A$. Also, since $A_1$ is a nonselfadjoint invertible matrix, $\Sc(A_1,A^*_1)$ is  simple if and only if $\Sc(A_1,A^*_1)$ is SI  (follows from Theorem \ref{inverse}). Furthermore, by Corollary \ref{cor3}, $\Sc(A_1,A^*_1)$ being SI implies that $A_1$ is a unitary matrix. Therefore, if $\Sc(A,A^*)$ is simple, it follows that $A$ is unitarily equivalent to $U\oplus 0$.

Conversely, if $A$ is unitarily equivalent to $U\oplus 0$, then $\Sc(U\oplus 0, U^*\oplus 0)$ forms a group and so also $\Sc(A, A^*)$ forms a group, and hence $\Sc(A, A^*)$ is simple.
	\end{proof}

We end this section by providing a characterization, solely via its norm, of a partial isometry when $\Sc(A, A^*)$ is an SI semigroup generated by a nonselfadjoint $A$.
This is for more general matrices (not necessarily Jordan matrices) 
(Theorem \ref{p} below). To prove this, we require the concept of $s$-numbers (singular number sequence) of a matrix.
The $s$-numbers of a matrix $A\in M_n(\mathbb C)$ is defined as the $n$-tuple of eigenvalues of diagonalized $|A|:= (A^*A)^{1\slash 2}$ arranged in decreasing order. So for instance the first $s$-number,
$s_1(A) = ||A||$.  

In general, if $A$ is a partial isometry, then $||A|| =1$, but the converse need not be true. For instance, 
\begin{equation*}
A = \begin{bmatrix}
0 & 1/2\\
1 & 0
\end{bmatrix}.
\end{equation*}
has $||A|| =1$, but $A$ is not a partial isometry because $A^*A$ is not a projection.

But for SI semigroups $\Sc(A, A^*)$, we prove in Theorem \ref{p} that if $||A|| =1$, then $A$ must be a partial isometry. 
This result can fail  in infinite dimensions (see Example \ref{infinite} below).

Let $\{s_j(A)\}^n_{j=1}$ denote the $s$-numbers of $A\in M_n(\mathbb{C})$.

\begin{proposition}\label{prop2.1}
For $A\in M_n(\mathbb{C})$, $A$ is a partial isometry if and only if $s_j(A)\in \{0,1\}$ for all $1\leq j\leq n$.
\end{proposition}
\begin{proof}
	If $A$ is a partial isometry, then $A^*A$ is a projection, so the eigenvalues of $A^*A$ are in the set $\{0,1\}$, whose square roots are also then the $s$-numbers of $A$. Conversely, if $s_j(A)\in \{0,1\}$ for all $1\leq j\leq n$. Then $|A|$ has all its eigenvalues contained in the set $\{0,1\}$. Moreover, since $|A|\geq 0$, so $|A|$ is unitarily diagonalizable. Therefore there exists a unitary matrix $U$ and a diagonal matrix $D$ such that $|A|=UDU^*$, where $D$ has all its diagonal entries in the set $\{0,1\}$. Since $D$ is a projection, $|A|$ is also a projection and hence its square $A^*A$ is also a projection, or equivalently, $A$ is a partial isometry.
\end{proof}

We recall here the matrix version of a set of $s$-number inequalities in Gohberg and Kre\u{\i}n \cite{GK} which we need for the lemma following.

\cite[Corollary 4.1]{GK} \textit{For any two matrices $A, B\in M_n(\mathbb{C})$ one has 
\begin{equation*}\label{GK}
\displaystyle{\sum^k_{j=1}}s_j(AB) \leq \displaystyle{\sum^k_{j=1}}s_j(A)s_j(B), \qquad (k= 1, \cdots, n).
\end{equation*}
And as the authors indicate, the inequality can naturally be generalized to the case of $m$ matrices $A_1, \cdots, A_m$.}
\begin{equation}\label{GK'}
\displaystyle{\sum^k_{j=1}}s_j(A_1 \cdots A_m) \leq \displaystyle{\sum^k_{j=1}}s_j(A_1) \cdots s_j (A_m), \qquad (k= 1, \cdots, n).
\end{equation}

Nonzero partial isometries always have norm one, but the converse clearly does not hold.
However, in the SI environment, it does.

\begin{lemma}\label{L}
For $A\in M_n(\mathbb{C})$ a nonselfadjoint matrix with $\Sc(A,A^*)$ an SI semigroup, 
if $||A||=1$ then $A$ is a partial isometry.
\end{lemma}
\begin{proof}
	Suppose $\Sc(A,A^*)$ is an SI semigroup. Then the principal ideal $(A)_{\Sc(A,A^*)}$ is selfadjoint, or equivalently as usual, $A^*=XAY$ for some $X,Y\in \Sc(A,A^*)\cup\{I\}$, where not both $X, Y$ are equal to the identity matrix because of the nonselfadjointness of $A$. Then the word $XAY$ is a word in powers of $A$ and $A^*$ with at least two terms. So by applying $s$-number inequality Equation (\ref{GK'}) to $XAY$ which is a finite product of powers of $A$ and $A^*$, and using the fact that $s_j(A^*)=s_j(A)$ for $1\leq j\leq n$, one obtains for some $m \ge 2$, 
	$$\sum_{j=1}^{k}s_j(A) = \sum_{j=1}^{k}s_j(A^*)= \sum_{j=1}^{k}s_j(XAY)\leq \sum_{j=1}^{k} s_j^m (A),\quad \text{for each}~1\leq k\leq n.$$
	Since $||A||=1,~s_1(A)=1$ and $0\leq s_j(A)\leq 1$ for all $1\leq j\leq n$ since $s$-numbers are in decreasing order. We will next show that $s_j(A)\in \{0,1\}$ for all $1\leq j\leq n$. 
	Suppose there exists some $1\leq j\leq n$ such that $0<s_j(A)<1$, and choose $r$ to be the smallest such index so that $0<s_r(A)<1$. Then from the above display, for $k=r$, one has
	$$\sum_{j=1}^{r}s_j(A)\leq \sum_{j=1}^{r} s_j^m (A).$$
	Since $s_j(A)=1$ for $1\leq j\leq r-1$, this further implies that 
	$$s_r(A)\leq s_r^m (A)$$
	contradicting $s_r^m (A)< s_r(A)$ because $0<s_r(A)<1$ and $m\geq 2$. Hence, $s_j(A)\in \{0,1\}$ for all $1\leq j\leq n$ and by Proposition~\ref{prop2.1}, $A$ is a partial isometry.
\end{proof}
Lemma \ref{L} immediately provides the following norm characterization of a partial isometry $A$ under the SI property of $\Sc(A, A^*)$:
\begin{theorem}\label{p}
For $A\in M_n(\mathbb{C})$ a nonselfadjoint matrix, let $\Sc(A,A^*)$ be an SI semigroup. Then
\begin{center}
$A$ is a partial isometry if and only if $||A|| =1$.
\end{center}
\end{theorem} 
We now give the example promised prior to Proposition \ref{prop2.1} that the above lemma may not hold for infinite rank operators. 
\begin{example}\label{infinite}
	Let $W$ be a weighted shift operator on the Hilbert space $l^2(\mathbb{N})$ with the weight sequence $(1/2,1,1,1,\ldots)$. Then $||W||=1$ and one can check that $W$ satisfies the relation $W^*W^2=W$. Therefore, it follows from Proposition \ref{quasi-isometries} proved below that $\Sc(W,W^*)$ is simple and hence trivially an SI semigroup. But $W$ is not a partial isometry because $W^*W$ is not a projection (as $W^*We_1=1/4e_1$).
\end{example}

We note that the following proposition with essentially the same proof also holds for operators $T \in B(\mathcal H)$.
\begin{proposition}\label{quasi-isometries}
	Let $T\in M_n(\mathbb C)$ and $T$ satisfies $(T^*T)T=T$. Then $\Sc(T,T^*)$ is simple. 
	\end{proposition}
\begin{proof}
	Since $(T^*T)T=T$, by induction one obtains, for all $n \geq 2$,
	\begin{equation}\label{equation1}
	{T^*}^nT^{n+1}=T
	\end{equation}
	Hence also, for all $n \geq 1$,
	\begin{equation}\label{equation2}
	{T^*}^{n+1}T^n=T^*
	\end{equation}
	Recall the semigroup list
	\noindent $\Sc(T,T^*) =  \{T^n, {T^*}^n,  \Pi_{j=1}^{k}{T^*}^{m_j}T^{n_j},
	 (\Pi_{j=1}^{k}{T^*}^{m_j}T^{n_j}){T^*}^{m_{k+1}},\Pi_{j=1}^{k}T^{n_j}{T^*}^{m_j},$  $(\Pi_{j=1}^{k}T^{n_j}{T^*}^{m_j})T^{n_{k+1}}\}$ where $n \ge 1,\,  k\ge1,\, n_j, m_j \ge 1\, \text{ for }\, 1 \le j \le k \text{ and } n_{k+1},m_{k+1}\geq 1$. To prove $\Sc(T,T^*)$ is simple, it suffices to show that the principal ideal generated by each form in the semigroup list coincides with the entire semigroup $\Sc(T,T^*)$. Furthermore, it suffices to show that the principal ideals generated by all the fourth and sixth forms coincide with $\Sc(T^*, T)$ because each principal ideal generated by each of the other forms contains a fourth and a sixth form.
	 
Consider a matrix $A$ in the fourth form. So $A=(\Pi_{j=1}^{k}{T^*}^{m_j}T^{n_j}){T^*}^{m_{k+1}}$ for some $m_j, n_j \geq 1$ and $m_{k+1} 
\geq 1$. Let $s = \sum_{j=1}^{k} n_j$ and $r = \sum_{j=1}^{k+1} m_j$. Then,
\begin{equation*}
\begin{aligned}
{T^*}^{s}A &= {T^*}^s({T^*}^{m_1}T^{n_1})(\Pi_{j=2}^{k}{T^*}^{m_j}T^{n_j}){T^*}^{m_{k+1}}\\
&= {T^*}^{s+ m_1 - n_1 -1} ({T^*}^{n_1 +1}T^{n_1})(\Pi_{j=2}^{k}{T^*}^{m_j}T^{n_j}){T^*}^{m_{k+1}} \quad \text{(add and substract $n_1+1$ from the power $s$ of $T^*$)}\\ 
&= {T^*}^{s+ m_1 - n_1}(\Pi_{j=2}^{k}{T^*}^{m_j}T^{n_j}){T^*}^{m_{k+1}} \qquad \qquad (\text{from Equation (\ref{equation2}) above ${T^*}^{n_1 +1}T^{n_1} = T^*$})\\
&= {T^*}^{s+(m_1-n_1)+(m_2 -n_2 -1)}({T^*}^{n_2 +1}T^{n_2})(\Pi_{j=3}^{k}{T^*}^{m_j}T^{n_j}){T^*}^{m_{k+1}}\\
&={T^*}^{s+(m_1-n_1)+(m_2-n_2)}(\Pi_{j=3}^{k}{T^*}^{m_j}T^{n_j}){T^*}^{m_{k+1}} \qquad (\text{again from Equation (\ref{equation2})})\\
\qquad \vdots\\
&= {T^*}^{\sum_{j=1}^{k} m_j}{T^*}^{m_{k+1}}\\
&= {T^*}^r \qquad (\text{from Equation (\ref{equation1}) above})
\end{aligned}
\end{equation*}
Since ${T^*}^{s+1}AT^r \in (A)_{\Sc(T, T^*)}$ and ${T^*}^{s+1}AT^r=T^*({T^*}^sA)T^r={T^*}^{r+1}T^r=T^*$ (from Equation \ref{equation2}), one obtains $T^* \in (A)_{\Sc(T, T^*)}$. Also note that $({T^*}^{s}A)T^{r+1} ={T^*}^rT^{r+1}= T$ (from Equation (\ref{equation1})), so $T\in (A)_{\Sc(T, T^*)}$. And since $T,T^*\in(A)_{\Sc(T, T^*)}$, $(A)_{\Sc(T, T^*)} = \Sc(T, T^*)$. 

We next consider the sixth form. So $A= (\Pi_{j=1}^{k}T^{n_j}{T^*}^{m_j})T^{n_{k+1}}$ for some $n_j, m_j \geq 1$, $1\leq j\leq k$, and $n_{k+1} 
\geq 1$. The matrix ${T^*}^{n_1}A{T^*}^{n_{k+1}} \in (A)_{\Sc(T, T^*)}$. Note that ${T^*}^{n_1}A{T^*}^{n_{k+1}}$ is back in the fourth form. Hence $({T^*}^{n_1}A{T^*}^{n_{k+1}})_{\Sc(T, T^*)}=\Sc(T, T^*)$. But $({T^*}^{n_1}A{T^*}^{n_{k+1}})_{\Sc(T, T^*)}\subset (A)_{\Sc(T, T^*)}$ so $(A)_{\Sc(T, T^*)}=\Sc(T, T^*)$. 
\end{proof}

\section{SI semigroups $\Sc(A, A^*)$ generated by matrices $A$ with nonnegative entries}

Our first attempts to progress beyond the characterizations regarding SI for rank one operators in \cite{PW21} were to investigate matrices (finite and infinite) with nonnegative entries. The results we obtain are more elementary than the results obtained earlier here and in [\cite{PWS}], so we present them here to complete this paper.

\begin{proposition}\label{prop1.1'}
	Let $A=[a_{ij}]$ and $B=[b_{ij}]$ denote nonzero matrices in $M_n(\mathbb{C})$ such that the nonzero entries of $A$ and $B$ are greater than $1$. Let $ a= \min\limits_{1\leq i,j \leq n}\{a_{ij}: a_{ij}\neq 0\}$ and $b=\min\limits_{1\leq i,j \leq n} \{b_{ij}: b_{ij}\neq 0\}$. If $AB=[c_{ij}]$ is nonzero, then
$$\min\limits_{1\leq i,j \leq n} \{c_{ij}: c_{ij}\neq 0\}> ab.$$
Consequently, $\min\limits_{1\leq i,j \leq n} \{c_{ij}: c_{ij}\neq 0\}> a$ and $\min\limits_{1\leq i,j \leq n} \{c_{ij}: c_{ij}\neq 0\}> b$.
	\end{proposition}
\begin{proof} Since the nonzero entries of $A$ and $B$ are greater than $1$, the nonzero entries of 
$AB := C$ are greater than $1$. 
Indeed, $C = [c_{ij}] = [\sum_{k=1}^{n} a_{ik}b_{kj}]$. 
Then each $c_{ij} \ne 0$ (if any) has $c_{ij} \ge a_{ik}b_{kj} \ge ab$ for some $1 \le k \le n$. 
$c_{ij}= \sum_{k=1}^{n} a_{ik}b_{kj}$. 
Then since $C$ is nonzero, $c_{i_0j_0}\neq0$ for some $1\leq i_0,j_0 \leq n$. And so there exists $1\leq s \leq n$ such that $a_{i_0s}b_{sj_0}\neq 0$.
Since $a_{ij},b_{ij}\geq 0$ for $1\leq i,j \leq n$, 
 $${c_{i_0j_0}=\sum_{k=1}^{n}a_{i_0k}b_{kj_0} \geq a_{i_0s}b_{sj_0}} \geq ab.$$
Therefore, $\min\limits_{1\leq i,j \leq n} \{c_{ij}: c_{ij}\neq 0\}\geq ab$.
Since the nonzero entries of $A$ and $B$ are greater than $1$, $a, b >1$ and hence $ab >a$ and $b$. 
\end{proof}


One can easily extend Proposition \ref{prop1.1'} to finite products of finite matrices.
\begin{corollary}\label{cor2'}
For each $1\leq k\leq m, let A_k = [a_{ij}^{(k)}]$ denote a matrix in  $M_n(\mathbb{C})$ with nonzero entries of each $A_k$ greater than $1$.
If $A=\prod_{k=1}^{m}A_k = [a_{ij}]$ is nonzero, then for $1\leq k\leq m$, we have for each $k$,
$$\min\limits_{1\leq i,j \leq n} \{a_{ij}: a_{ij}\neq 0\} \ge 
\Pi_{l = 1}^m \min\limits_{1\leq i,j \leq n}\{a_{ij}^{(l)}: a_{ij}^{(l)}\neq 0\} 
> \min\limits_{1\leq i,j \leq n}\{a_{ij}^{(k)}: a_{ij}^{(k)}\neq 0\}.$$
\end{corollary}

\begin{rem}
Proposition \ref{prop1.1'} may not hold for infinite matrices. Consider $A=B=diag(1+1/n)_{n=1}^\infty$. Then $1 = \inf \{a_{ii}^2 : i\in \mathbb{N}\}=\inf\{(1+1/n)^2: n\in \mathbb{N} \}=\inf\{(1+1/n): n\in \mathbb{N} \}=\inf \{a_{ii} : i\in \mathbb{N}\}$.
\end{rem}
	\begin{corollary} \label{1.1 cor'}
		If $A=[a_{ij}]$ is a matrix in $M_n(\mathbb{C})$ with all nonzero entries greater than 1 (or all less than $-1$), then $\Sc (A,A^*)$ is a non-SI semigroup.
	\end{corollary}
	\begin{proof} If $\Sc (A,A^*)$ is an SI semigroup, then the principal ideal $(A)_{\Sc (A,A^*)}$ of $\Sc (A,A^*)$ is selfadjoint. That is, $A^*=XAY$ where $X,Y \in {\Sc (A,A^*)}\cup\{I\}$ and $X$ and $Y$ cannot both be the identity (as $A^*\neq A$). By Corollary \ref{cor2'} applied to the product $XAY$, the minimum nonzero entry of $XAY$ is greater than the minimum nonzero entry of $A$ (or $A^*$). This contradicts the fact that the minimum nonzero entry of $XAY$ must be equal to the minimum nonzero entry of $A^*$ because of the equality $A^*=XAY$. Therefore the ideal  $(A)_{\Sc (A,A^*)}$ is not selfadjoint. Hence, $\Sc (A,A^*)$ is a non-SI semigroup. And similarly if all entries are less than $-1$.
		\end{proof}
	
		

	The proof of the following version combining Proposition \ref{prop1.1'} and Corollary \ref{cor2'} but for infinite matrices is straightforward and is left to the reader. And likewise for Corollary \ref{1.1 cor'}.
	\begin{proposition}\label{prop1.2'}
		Let for each $1\leq k\leq m, let A_k = [a_{ij}^{(k)}]$ denote matrix representations of some operators in $\mathcal{B}(\mathcal{H})$ in a common orthonormal basis such that the nonzero entries of each $A_k$ are greater than 1.\\
		If $A:=\prod_{k=1}^{m}A_k = [a_{ij}]$ is nonzero, then 
		$\inf\limits_{i,j} \{a_{ij}: a_{ij}\neq 0\} \geq \prod_{k=1}^{m}\{\inf\limits_{i,j} \{a_{ij}^{(k)}: a_{ij}^{(k)}\neq 0\}\}.$\\
		Moreover if for all $1\leq k\leq m, \inf\limits_{i,j} \{a_{ij}^{(k)}:a_{ij}^{(k)}\neq 0\}> 1, $ then\\
		$\inf\limits_{i,j} \{a_{ij}: a_{ij}\neq 0\} \ge 
\Pi_{l = 1}^m \inf\limits_{i,j}\{a_{ij}^{(l)}: a_{ij}^{(l)}\neq 0\}  > \inf\limits_{i,j} \{a_{ij}^{(k)}: a_{ij}^{(k)}\neq 0\}$ ; for all $1\leq k\leq m$.
		\end{proposition}
	\begin{corollary}\label{cor3'}
			Let $A=[a_{ij}]$ be a matrix representation of an operator in $\mathcal{B}(\mathcal{H})$ with respect to some orthonormal basis such that the nonzero entries of $A$ are greater than $1$.
			If	$\inf\limits_{i,j} \{a_{ij}: a_{ij}\neq 0\} >1$, then  $\Sc (A,A^*)$ is a non-SI semigroup. And similarly for the $-1$ case.
		\end{corollary}

\section{Data availability}
Data sharing not applicable to this article as no datasets were generated or analyzed during the current
study.
\section{Declarations}
The first author was supported by Science and Engineering Research Board, Core Research Grant
002514. The last author was partially supported by Simons Foundation collaboration grants 245014 and
636554. The authors have no other conflicts of interest to report.

\end{document}